\documentclass[a4paper,11pt]{article}
\usepackage{amsmath,amssymb,amsthm}
\usepackage{graphicx}
\usepackage{appendix}
\usepackage{lmodern}
\usepackage{listings}
\usepackage{url}
\usepackage{wrapfig}
\usepackage{float}
\usepackage{stmaryrd}
\usepackage{dsfont}
\usepackage{esint} 
\usepackage{fancyhdr}
\usepackage[notlof, notlot]{tocbibind}
\usepackage[utf8]{inputenc}
\usepackage[T1]{fontenc}
\usepackage{enumerate}
\usepackage[dvipsnames]{xcolor}
\usepackage[english]{babel}
\usepackage[12pt]{moresize}
\usepackage{hyperref}
\usepackage{authblk}

\DeclareTextSymbol{\degre}{T1}{6}

\definecolor{Change}{RGB}{30, 59, 105}
\definecolor{Change2}{RGB}{17, 82, 28}

\newcommand{\IN}{\mathbb{N}}
\newcommand{\IZ}{\mathbb{Z}}
\newcommand{\IR}{\mathbb{R}}

\newcommand{\IT}{\mathbb{T}}
\newcommand{\IE}{\mathbb{E}}
\newcommand{\dst}{\displaystyle}
\newcommand{\dd}{\mathrm{d}}
\newcommand{\Div}{\mathrm{div}}
\newcommand{\po}{\left(}
\newcommand{\pf}{\right)}
\newcommand{\nv}[1]{{\color{black}#1}}

\newtheorem{theo}{Theorem}

\newtheorem{prop}{Proposition}
\newtheorem{cor}{Corollary}
\newtheorem{lem}{Lemma}

\newtheorem{rem}{Remark}
\newtheorem{assu}{Assumption}

\title{The Adaptive Biasing Force algorithm with non-conservative forces and related topics}

\author[1]{\small Tony Lelièvre}
\author[2, 3]{\small Lise Maurin}
\author[2, 3]{\small Pierre Monmarché}

\affil[1]{\footnotesize Université Paris-Est, CERMICS (ENPC), Inria, 77455 Marne-la-Vallée, France; \texttt{tony.lelievre@enpc.fr}}
\affil[2]{\footnotesize Sorbonne Université, LJLL, 4 place Jussieu, 75005 Paris, France; \texttt{lise.maurin,pierre.monmarche@upmc.fr}}
\affil[3]{\footnotesize Sorbonne Université, LCT, 4 place Jussieu, 75005 Paris, France}

\begin{document}
\maketitle

\abstract{We propose a study of the Adaptive Biasing Force method's robustness under generic (possibly non-conservative) forces. We first ensure the flat histogram property is satisfied in all cases. We then introduce a fixed point problem yielding the existence of a stationary state for both the Adaptive Biasing Force and Projected Adapted Biasing Force algorithms, relying on generic bounds on the invariant probability measures of homogeneous diffusions. Using classical entropy techniques, we prove the exponential convergence of both biasing force and law as time goes to infinity, for both the Adaptive Biasing Force and the Projected Adaptive Biasing Force methods.\\ }

\section{Introduction}
\label{sec:introduction}
After presenting in Sections \ref{subsec:setting}, \ref{subsec:metastability} and \ref{subsec:ABF} the motivation and well-known results on the Adaptive Biasing Force (ABF) method applied to the overdamped Langevin dynamics with conservative forces, we present in Section \ref{subsec:non-conservative-case} the dynamics we are interested in, namely the ABF method applied to the overdamped Langevin dynamics with non-conservative forces.
\subsection{Setting}
\label{subsec:setting}

Let us work within the so-called \textit{canonical ensemble} (or \textit{NVT ensemble}), where a system of $N$ particles is contained in a fixed volume $\mathcal{V}$, and is in contact with a thermostat of constant temperature $T$. Denote by $q=(q_{1},\ldots,q_{N}) \in \mathcal{D}$ the positions, $p=(p_{1}, \ldots, p_{N}) \in \IR^{dN}$ the momenta, and $(m_{1}, \ldots, m_{N}) \in \IR^{N}$ the masses of the particles, where $\mathcal{D}$ is the configuration space and $d \in \{1,2,3\}$ is the space dimension. 
Usually, $\mathcal{D}$ is an open subset of $\IR^{dN}$ (or $\mathbb{T}^{dN}$, where  the $dN$-dimensional torus is viewed as the cube $[0,1]^{dN}$ with opposite sides identified, in other words, $\mathbb{T}^{dN}=\IR^{dN}/\IZ^{dN}$). Interactions between particles are taken into account via a potential function $V : \mathcal{D} \rightarrow \IR$, so that the system's total energy is given by the following Hamiltonian:

\begin{equation*}
H(q,p)=V(q)+ \frac{1}{2}p^{\top}M^{-1}p, 
\end{equation*}
with $M=\mathrm{diag}(m_{1}I_{d}, \ldots, m_{N}I_{d})$ being the mass matrix. Since this Hamiltonian is \textit{separable}, the positions and the momenta are independent variables in the canonical ensemble, namely under the probability distribution $Z^{-1}e^{-\beta H(q,p)} \, \dd q \dd p$ where $\beta = 1/(k_{B}T)$, with $k_B$ being the Boltzmann constant, and $Z=\int_{\mathcal{D} \times \IR^{dN}}e^{-\beta H(q,p)} \, \dd q \dd p  $ is the normalization constant, or \textit{partition function}. The momenta $p$ being distributed according to a Gaussian measure, the main issue resides in sampling the positions $q$, which are distributed according to the \textit{Boltzmann-Gibbs measure}:
 \begin{equation*}
 \mu(\dd q)= Z_{\mu}^{-1}e^{-\beta V(q)} \dd q, \qquad Z_{\mu}=\dst{\int_{\mathcal{D}} e^{-\beta V(q)} \, \dd q } .
\end{equation*}  
Thermodynamic properties are obtained by averaging functions of the microstate $q$ which are called \textit{observables}. Given an observable $\psi$, one would like to compute the following thermodynamic quantity:
\begin{equation*}
 \mathbb{E}_{\mu}[\psi] = \displaystyle{\int_{\mathcal{D}} \psi \, \dd \mu }. 
 \end{equation*}
One of the simplest dynamics to sample the Boltzmann-Gibbs measure is the \textit{overdamped Langevin dynamics}:
 \begin{equation}
 \dd X_{t}=-\nabla V(X_{t}) \, \dd t +\sqrt{2\beta^{-1}} \dd W_{t}, \label{GeneralOverdampedLangevin}
 \end{equation}
where $(W_{t})_{t\geq 0}$ is a $dN$-dimensional standard Brownian motion, and $-\nabla V: \mathcal{D} \rightarrow \IR^{dN}$ is the \textit{interaction force}. Notice that here, the interaction force is \textit{conservative}, namely it is the gradient of a function (here, minus the gradient of the potential energy $V$). Under reasonable assumptions on the potential $V$ (see \cite{LRS07} for more details), the process $(X_{t})_{t\geq 0}$ is ergodic with respect to $\mu$. In other words, for any observable $\psi \in \mathcal{C}^{\infty}_{0}(\mathcal{D})$, the average over a trajectory of the process converges to the canonical average:
 \begin{equation}\label{eq:ergodicity}
 \lim\limits_{\tau \rightarrow +\infty} \frac{1}{\tau}\displaystyle{\int_{0}^{\tau} \psi(X_{t}) \, \dd t } =\mathbb{E}_{\mu}[\psi]. 
 \end{equation}

\subsection{Metastability, reaction coordinate and free-energy profiles}
\label{subsec:metastability}

Computing thermodynamic averages can be troublesome, as microscopic and macroscopic timescales can violently differ. Typical microscopic phenomena occur on timescales of the order of $10^{-15}$s, while macroscopic ones can take up to $1$ h \cite{ActaNumerica}. Furthermore, $N$ needs to be sufficiently large so that the targeted macroscopic phenomena can emerge from the collective, microscopic behaviour of the system.

Such timescales differences are linked to the system's \textit{metastability}: low-energy regions of the configuration space are separated by either high-energy or high-entropy barriers. These regions are called metastable: the process \eqref{GeneralOverdampedLangevin} remains trapped in a metastable region and occasionally jumps to another one after a long period of time. 
From a probabilistic point of view, metastability is linked to the \textit{multimodality} of the measure $\mu$: likely regions are separated by low probability regions. The exploration of the state space by the process and the convergence of the trajectorial averages \eqref{eq:ergodicity} can thus take a considerably long time.

One way of avoiding metastability is to capture some slow components of the dynamics $(X_{t})_{t \geq 0}$. To do so, we consider \textit{transition coordinates} (also called reaction coordinates or collective variables), namely mappings $\xi : \mathcal{D} \rightarrow \mathcal{M}$, 
where $\mathcal M$ is a manifold of dimension $m \ll dN$. Transition coordinates are designed to provide a \textit{coarse-grained information} on the system's state (for example, the dihedral angle of a molecule, in which case $\mathcal{M}=\mathbb{T}$, or the signed distance to a hypersurface of $\mathcal{D}$, in which case $\mathcal{M}= \IR$).  
In other words, $\xi(x)\in\mathcal M$ is the macroscopic state of a microscopic state $x\in\mathcal D$. Designing a good reaction coordinate is a difficult problem, that will not be discussed further in the present work (see \cite{RevueMLMD} for a recent review on the question of automatic learning of transition coordinates).

Decomposing
 $$\mathcal{D}=\bigsqcup\limits_{z\in \mathcal{M}} \Sigma_{z} =\bigsqcup\limits_{z\in \mathcal{M}} \{q\in \mathcal{D} | \xi(q)=z\},$$
and denoting by $\sigma_{\Sigma_{z}}$ the measure on $\Sigma_{z}$ induced by the Lebesgue measure on $\mathcal{D}$, one can define the measure $ \delta_{\xi(q)-z}(\dd q )$ by
$$\delta_{\xi(q)-z}(\dd q ) = \frac{1}{\sqrt{\det G(q)}}  \sigma_{\Sigma_{z}}(\dd q),$$
where $G=\left(\nabla \xi \right)^{\top} \nabla \xi$, in other words,
$$G_{i,j}=\nabla \xi_{i} \cdot \nabla \xi_{j}, \quad \text {for all } (i,j) \in \llbracket 1,m \rrbracket^{2}.$$
The free energy associated to $\xi$ is then expressed as follows: for every $z\in\mathcal M$,
\begin{equation}
A(z)=\displaystyle{-\frac{1}{\beta} \ln (Z_{\Sigma_{z}}),\quad Z_{\Sigma_{z}}=\int_{\Sigma_{z}} e^{-\beta V(q)} \delta_{\xi(q)-z}(\dd q )}, \label{eq:free-energy}
\end{equation}
assuming $V$ and $\xi$ are such that $Z_{\Sigma_{z}}<+\infty$. As can be seen using the co-area formula \cite{LRS07}, this definition ensures that the image of $\mu$ by $\xi$ is given by
\begin{equation}
\xi * \mu \,(\dd z)=\frac{e^{-\beta A(z) }\dd z}{\displaystyle{\int_{\mathcal{M}} e^{-\beta A(u)} \dd u}}.\label{BGXi}
\end{equation}

\subsection{The Adaptive Biasing Force method}
\label{subsec:ABF}

Introducing a reaction coordinate allows us to construct a  less metastable dynamics, the idea being to substitute the potential $V$  in \eqref{GeneralOverdampedLangevin} for a \textit{biased potential} $V-A\circ\xi$. The new equilibrium measure is then
\begin{equation}\label{eq:mu-FE}
\mu_{A}(\dd q )=Z_{\mu_{A}}^{-1} e^{-\beta (V-A\circ \xi)(q)} \, \dd q,
\end{equation}
 where $Z_{\mu_{A}}= \int_{\mathcal{D}} e^{-\beta (V-A\circ \xi)(q)} \dd q$. Given the expression \eqref{BGXi}, the image of $\mu_{A}$ by $\xi$  is the uniform measure: $\xi * \mu_{A}=\lambda(\mathcal{M})^{-1}\mathds{1}_{\mathcal{M}}$, with $\lambda(\mathcal{M})$ being the Lebesgue measure of $\mathcal{M}$ (which is here assumed to be compact). Since, contrary to the initial probability measure $\xi * \mu$, the uniform measure is no longer multimodal, we expect a faster sampling of the phase space, provided $\xi$ is well chosen so that $\mu_{A}$ is less multimodal than $\mu$.

Although this change of potential can  accelerate the phase space sampling, the free-energy $A$ is \textit{a priori} unknown. The main idea to get round this issue will be to approximate on the fly $A$ , or $\nabla A$, its derivative with respect to the reaction coordinate. To do so, we will consider the \textit{Adaptive Biasing Force (ABF) algorithm} \cite{Darve-Pohorille,Henin-Chipot}:
\begin{equation}
\left\{\begin{array}{rcl}
 \dd X_{t}& =&  \left(-\nabla V(X_t) +B_{t}\left(\xi(X_t)\right) \nabla \xi(X_t)  \right)  \, \dd t +\sqrt{2\beta^{-1}} \dd W_{t} \\
 B_{t}(z)&=&\mathbb{E}[F(X_{t}) \,| \,\xi(X_{t})=z]  \qquad \forall z \in \mathcal{M},
 \end{array}\right.  \label{GeneralABF}
 \end{equation}
where $-\nabla V$ is the conservative interaction force, and $F$ is the so-called \textit{local mean force}, which is the vector with components $(F_{i})_{i \in \llbracket 1,m \rrbracket}$ given by:
\begin{equation*}
F_{i}= \sum_{j=1}^{m} G_{i,j}^{-1} \nabla \xi_{j} \cdot \nabla V -\beta^{-1} \Div \left(\sum_{j=1}^{m} G_{i,j}^{-1} \nabla \xi_{j}\right),
\end{equation*}
where $G_{i,j}^{-1}$ denotes the $(i,j)$-component of the inverse of the matrix $G$ defined above.
 This process is motivated by the fact that the aforementioned free energy satisfies:
 \[\nabla A(z) = \mathbb{E}[F(X) | \xi(X)=z], \quad \forall z \in \mathcal{M}\qquad \text{ if }X\sim \mu_A\,,\]
 so that $\mu_A$ is a fixed point of the Fokker-Planck equation associated to the process. In other words, if $X_0\sim\mu_A$, then $X_t\sim \mu_A$ for all $t\geqslant 0$ and $(X_t)_{t\geqslant 0}$ is exactly the diffusion \eqref{GeneralOverdampedLangevin} with the biased potential $V-A\circ\xi$.

Starting from another initial distribution, using entropy estimates and functionnal inequalities, it has been proven in \cite{LRS07}, under mild assumptions, that this fixed point is in fact an attractor of the dynamics, in the sense that $B_t$ converges  to $\nabla A$ in the long-time limit, and the law of $X_t$ converges to $\mu_A$.\\
 
\begin{rem}\

 \begin{enumerate}
\item[]$\triangleright$ In some cases $\mathcal{M}$ is not bounded, for example when $\xi$ is a distance. If so, an additional confining potential $W\circ \xi$ is  needed in the drift \cite{LRS07}.
\item[]$\triangleright$ As discussed in \cite{LRS07}, the algorithm \eqref{GeneralABF} can be modified in order to obtain a diffusive behaviour for the law of $\xi(X_{t})$. Additional terms depending on $\xi$ are added to obtain the following variant:
\begin{equation*}
\left\{\begin{array}{rcl}
 \dd X_{t} & = & \left(-\nabla V +B_{t}\circ \xi -\nabla W\circ \xi + \beta^{-1} \nabla\ln (|\nabla \xi |^{-2})\right)|\nabla \xi |^{-2} (X_{t})\,  \dd t +\sqrt{2\beta^{-1}} |\nabla \xi |^{-1}(X_{t}) \dd W_{t} \\
B_{t}(z) &= &\mathbb{E}[F(X_{t}) \,| \,\xi(X_{t})=z],  \quad \forall z \in \mathcal{M}.
 \end{array}\right. 
 \end{equation*}

In this case the longtime convergence of $B_{t}$ towards $\nabla A$ is stronger than in the case of \eqref{GeneralABF}, in that it requires less hypothesis. 
\end{enumerate}
\end{rem}

We might also consider a variant of the ABF method, namely the \textit{Projected Adaptive Biasing Force} (PABF) algorithm, introduced in \cite{Alrachid}:
$$\left\{\begin{array}{rcl}
 \dd X_{t}&=& \left(-\nabla V(X_t) +B_{t}\left(\xi(X_t)\right) \nabla \xi(X_t)  \right)  \, \dd t +\sqrt{2\beta^{-1}} \dd W_{t} \\
 B_{t} &=&\mathsf{P}_{L^{2}(\lambda)}\left(G_t\right)\\
 G_t(z) & = & \mathbb{E}[F(X_{t}) \,| \,\xi(X_{t})=z]  \qquad \forall z \in \mathcal{M},
 \end{array}\right.  $$
where $\mathsf{P}_{L^{2}(\lambda)}(f)$ stands for the Helmholtz projection with respect to the Lebesgue measure $\lambda$ of a vector field $f$ on an open bounded set $\mathcal{M}\subset \IR^{dN}$ with Lipschitz boundary $\partial\mathcal{M}$ \cite{Ambrosio}. In other words, it is the gradient of the minimizer on $\{g\in H^1(\mathcal{M}),\   \int_{\mathcal{M}} g \dd x  = 0 \}$ of
 $$g\mapsto \int_{\mathcal{M}} |f(x)-\nabla g(x)|^2 \dd x\,.$$
 More generally, if $\nu$ is a continuous positive measure on $\mathcal{M}$, the Helmholtz projection with respect to $\nu$ is the minimizer on $\{g\in H^1(\mathcal{M}),\   \int_{\mathcal{M}} g \dd x  = 0 \}$ of $g\mapsto \int_{\mathcal{M}} |f(x)-\nabla g(x)|^2 \nu(\dd x)$.
 
\subsection{The non-conservative case}
\label{subsec:non-conservative-case}

From now on, we only consider periodic boundary conditions and reaction coordinates that are Euclidean coordinates of the system, namely $\mathcal D = \mathbb T^n=\IR^{n}/\IZ^{n}$ for some $n\in \mathbb N^{*}$, $\mathcal M = \mathbb T^m$ for $m \in \IN^*$ such that $m\leqslant n$ and $\xi(x,y) = x$, where we decompose $(x,y)\in \mathcal D$ with $x\in\mathbb T^m$ and $y\in\mathbb T^{n-m}$. This latter restriction may seem  quite narrow: nevertheless, it is the generic case used for alchemical reactions \cite{KongXianjun}. Besides, more general reaction coordinates can be reduced to this setting by adding extended variables \cite{ChipotEABF}. Here, such restriction is made only for the sake of clarity: most arguments could be extended (at the price of heavier computations) to the general case $\xi(x,y) \in \mathcal{M}$.

We are interested in the case where the force in \eqref{GeneralOverdampedLangevin} is not necessarily conservative, namely is not the gradient of some potential energy $V$. There are several motivations for this approach, one of them being that the numerical computation of conservative forces $-\nabla V$ sometimes relies on approximations which make the force \textit{a priori} not conservative, in particular in the context of \textit{ab initio} molecular dynamics, see e.g. \cite{PulayFogarasi2004,Niklasson2006,Tkatchenko2017}. In this case, one is interested in knowing if, by controlling the error made on the force $-\nabla V$, one can deduce an estimation of the error made on the system's free energy. The robustness of a diffusion's invariant measure with respect to the perturbation of its drift is a classical problem (see e.g. Section~\ref{subsec:perturb-diffusion}), but note that in the ABF case, the adaptive procedure makes the question more subtle. Moreover, the convergence of the ABF method in such a context cannot be deduced from the aforementionned convergence analysis. We consequently consider the ABF algorithm in the case where $-\nabla V$ is replaced by a general force field $\mathcal F\in\mathcal C^1( \mathcal D ,\mathbb R^n)$ that we rewrite as~$\mathcal F(x,y) = (\mathcal F_1(x,y),\mathcal F_2(x,y))\in\mathbb R^m \times \mathbb R^{n-m}$. The local mean force is simply $F=-\mathcal F_1$, and the corresponding process is thus, for all $t \ge 0$:
\begin{equation}
\left\{\begin{array}{lclr}
 \dd X_{t} & = & \mathcal F_1(X_t,Y_t) \dd t + B_t(X_t) \dd t + \sqrt{2\beta^{-1}} \dd W_{t}^{1} &\\
  \dd Y_{t} & = & \mathcal F_2(X_t,Y_t) \dd t + \sqrt{2 \beta^{-1}} \dd W_{t}^{2} &
   \end{array}\right.  \label{F-ABF}
 \end{equation}
 where $W=(W^1,W^2)$ is a standard Brownian motion on $\IT^{m}\times \IT^{n-m}$, and, given the average mean force
 $$G_{t}(x) = -\mathbb{E}[\,\mathcal F_1(X_t,Y_t) \,| \, X_t=x\,],\quad  \forall t \ge 0, \forall x\in\mathbb T^m,$$ one has for all $t \ge 0$ and $x \in \IT^m$, in the case of the ABF method, 
 \begin{equation*}
 \quad B_{t}(x) = G_{t}(x),  
\end{equation*}
or, in the case of the PABF method,
\begin{equation*}
 B_{t}(x) = \mathsf{P}_{L^{2}(\lambda)}(G_t)(x):=\nabla H_{t}(x). 
\end{equation*} 
In either case, denoting by $\pi_t$ the law of  $Z_t = (X_t,Y_t)$ and $\pi_{t}^{\xi}(x)=\int_{\IT^{n-m}} \pi_t(x,y)\dd y$ the density of $X_{t} = \xi(Z_t)$, then
$$G_{t}(x) = \dst{\int_{\IT^{n-m}} -\mathcal F_1(x,y) \frac{\pi_t(x,y)}{\pi_t^\xi(x)} \dd y},$$
so that  $\pi_t$   is a weak solution of the Fokker-Planck equation associated to \eqref{F-ABF}, that is
\begin{equation}\label{EqFokkerPlanckABF}
\left\{\begin{array}{rcl}
\partial_t \pi_t &=&\beta^{-1} \Delta \pi_t - \nabla \cdot \po \mathcal F\, \pi_t \pf - \nabla_{x}  \cdot \po B_t \, \pi_t \pf\\
B_t & = & \left\{\begin{array}{ll}
G_t & \text{in the ABF case} \\
\nabla H_t \ =\ \mathsf{P}_{L^{2}(\lambda)}(G_t) & \text{in the PABF case}
\end{array}\right.\\
G_{t}(x) &=&  \int_{\IT^{n-m}} -\mathcal F_1(x,y) \frac{\pi_t(x,y)}{\pi_t^\xi(x)} \dd y \qquad \forall x\in\mathbb T^m.
\end{array}\right.
\end{equation}

For a given initial condition $\pi_{0}$, the existence of the process and the proof that it admits a density with respect to the Lebesgue measure, being a strong solution of (\ref{EqFokkerPlanckABF}), can be established by fixed point arguments or by the convergence of an interacting particles system \cite{JourdainLelievreLeroux}. We will not address this question here. As a consequence, we would like to emphasize that our arguments will be partially formal, in the sense that we work under the assumption that a density $\pi_t$ that solves \eqref{EqFokkerPlanckABF} exists and is sufficiently regular so that the algebraic computations in the proofs are valid.

Let us emphasize that the bias $B_t$ in \eqref{EqFokkerPlanckABF} (i.e. either $G_t$ or $\nabla H_{t}=\mathsf{P}_{L^{2}(\lambda)}(G_t)$) depends on $\pi_t$, which makes \eqref{EqFokkerPlanckABF} a non-linear PDE.

\begin{rem} In the conservative case, where $\mathcal{F}=-\nabla V$, and $\mu \varpropto e^{-\beta V}$, up to an additive constant, the free energy $A$ is characterized by either one of these properties:
\begin{enumerate}
\item $\xi * \mu \propto e^{-\beta A}$ (distribution of the reaction coordinate at equilibrium). 
\item $\nabla A = \mathbb{E}[\nabla_{1}V(Z) | \xi(Z)=\cdot \,]$ with $Z\sim \mu$ (average local mean force at equilibrium).
\item $\nabla A = \mathbb{E}[\nabla_{1}V(Z) | \xi(Z)=\cdot \,]$ with $Z\sim \mu_A $ (fixed point of the ABF algorithm).
\end{enumerate}
In the non-conservative case, there is no reason for these various definitions to coincide. Besides, $x\mapsto \mathbb{E}[-\mathcal{F}_{1}(Z) | \xi(Z)=x\,]$ is \textit{a priori} not a gradient. Denoting by $\mu_{\mathcal F}$ the invariant measure of the non-biased, out-of-equilibrium dynamics $\partial_t \pi_t = \beta^{-1} \Delta \pi_t - \nabla \cdot \po \mathcal F \pi_t \pf$, we are then led to consider the (in general different) functions $H_1$, $H_2$ and $H_3$ given, up to an additive constant, by
\begin{enumerate}
\item $\xi * \mu_{\mathcal F} \propto e^{-\beta H_1}$. 
\item $\nabla H_2 = \mathsf{P}_{L^2(\lambda)} \po \mathbb{E}[-\mathcal{F}_{1}(Z) | \xi(Z)=\cdot \, ]\pf$ with $Z\sim \mu_{\mathcal F}$.
\item $\nabla H_3 = \mathsf{P}_{L^2(\lambda)}\po \mathbb{E}[-\mathcal{F}_{1}(Z) | \xi(Z)=\cdot \, ]\pf$ with $Z\sim \pi_\infty^{\mathcal F}$  an equilibrium of the (P)ABF algorithm.
\end{enumerate}
In other words, in the non-conservative case, an equilibrium of an adaptive algorithm yields an alternative generalization of the notion of free energy that does not coincide in general with the log-density of the law of the reaction coordinates at (unbiased) equilibrium, and whose gradient is not in general the average local mean force at (unbiased) equilibrium.
\end{rem}

\noindent \textbf{Outline of this paper.} Section~\ref{sec:main-results} introduces several preliminary notions, before stating the main results. Section~\ref{sec:law-pi-xi} focuses on the law $\pi_{t}^{\xi}$ of the process $\left(X_t\right)_{t\ge 0}=\left(\xi(Z_t)\right)_{t \geq 0}$. More precisely, we show that $\pi_{t}^{\xi}$ satisfies a particular Fokker-Planck equation, which differs depending on the method considered, and that $\pi_{t}^{\xi}$ converges in the long-time limit towards the Lebesgue measure $\lambda$. Section~\ref{sec:existence-equilibrium} then states several results on the invariant measure of a generic diffusion, in order to adress the issue of the existence of both stationary measure and stationary biais to equation \eqref{EqFokkerPlanckABF}, and later handles the robustness of the conservative equilibrium to non-conservative perturbations. Eventually, Section~\ref{sec:long-time-cv} is devoted to the long-time convergence of both the ABF and PABF methods, in the conservative case, with a force $\mathcal{F}=-\nabla V$ (generalizing in particular results from \cite{Alrachid}), and in the non-conservative case, with a generic force $\mathcal{F}$.


\section{Main results}
\label{sec:main-results}

\subsection{Relative entropy and preliminary inequalities}
\label{subsec:def-entropies}

Let us first introduce several tools that will be used in the following. For $\mu,\nu$ two probability measures on the same space, we will denote by $\mu \ll \nu$ the absolute continuity of $\mu$ with respect to $\nu$. Now consider the relative entropy of $\mu$ with respect to $\nu$:
\[\mathcal H(\mu|\nu) = \left\{
\begin{array}{ll}
\dst{\int \ln \po \frac{\dd \mu}{\dd\nu}\pf \dd \mu} & \text{if }\mu\ll \nu,\\
+\infty & \text{otherwise.}
\end{array}
\right.\]
 Recall the Csisz\`ar-Kullback inequality:
\begin{equation}
\|\mu-\nu\|_{TV} \leq \sqrt{2 \mathcal{H}(\mu | \nu)}\,, \label{CK}
\end{equation}
where $\|\cdot\|_{TV}$ stands for the total variation norm. In particular, while the relative entropy is not a distance (it lacks the symmetry property), its convergence towards zero implies the convergence in total variance norm of $\mu$ towards $\nu$.

Similarly, let us define the Fisher information: for $\mu \ll \nu$,
$$I(\mu|\nu) = \displaystyle{\int |\nabla \ln \left( \frac{\dd \mu}{\dd \nu} \right)|^{2} \, \dd \mu}.$$
The probability measure $\nu$ is said to satisfy a Logarithmic Sobolev Inequality  $LSI(\rho)$ of constant $\rho>0$ if:
$$\forall \mu \ll \nu, \qquad \mathcal{H}(\mu|\nu) \leq \frac{1}{2\rho}I(\mu|\nu).$$
From \cite{Otto2000}, if $\nu$ satisfies a log-Sobolev inequality with constant $\rho>0$, then it also satisfies the so-called Talagrand inequality $\mathcal{T}(\rho)$ with constant $\rho>0$:
\begin{equation}
\forall \mu \ll \nu, \qquad W_{2}^{2}(\mu,\nu) \leq  \frac{2}{\rho}\mathcal{H}(\mu|\nu),\label{eq:Talagrand-ineq}
\end{equation}
where $W_{2}(\mu,\nu)$ is the Wasserstein distance with quadratic cost between the probability measures $\mu$ and $\nu$. More precisely, if $\mu$ and $\nu$ are defined on a general Riemannian manifold~$\Omega$:
$$W_{2}^{2}(\mu, \nu) = \displaystyle{ \inf\limits_{\pi \in \Pi(\mu,\nu)} \int_{\Omega \times \Omega} \omega(x,y)^{2} \, \dd \pi(x,y)},$$
where $\omega$ is the geodesic distance on $\Omega$, and $\Pi(\mu,\nu)$ is the set of coupling probability measures, i.e probability measures on $\Omega \times \Omega$ whose marginals are $\mu$ and $\nu$ respectively.

In the following, we will slightly abuse notations and denote $I(\mu |\nu)$, $\mathcal{H}(\mu |\nu)$ or $W_{2}(\mu|\nu)$ both in the case where $\mu$ and $\nu$ are probability measures, or probability density functions.

%
\subsection{Precise statements of the results}
\label{subsec:results}

In all this section, $\pi_t$ satisfies \eqref{EqFokkerPlanckABF}. First of all, let us consider the equation satisfied by the density $\pi_{t}^{\xi}$ in the general case where $\mathcal{F}$ is either conservative or non-conservative.

\begin{lem}\label{Lemme-loi-pi^xi}
The density $\pi_{t}^\xi$ satisfies the following Fokker-Planck equation:
\begin{equation}\partial_t \pi_t^\xi \ = \ \Delta \pi_t^\xi - \nabla \cdot \po (B_t-G_{t}) \pi_t^\xi\pf.  \label{FP-Pi_xi}
\end{equation}
\end{lem}
\begin{proof}
Take a test function $\varphi\in\mathcal C^\infty (\mathbb T^m)$. Then, using an integration by parts,
\begin{align*}
\frac{\dd }{\dd t} \int_{\mathbb T^m} \varphi \pi_t^\xi &  = \frac{\dd }{\dd t} \int_{\mathbb T^n} \varphi(x) \pi_t(x,y)\dd x\dd y\\
& =  \int_{\mathbb{T}^{n}} \po \Delta_x \varphi(x) + (\mathcal F_1(x,y) + B_t(x)) \nabla_x \varphi(x) \pf \pi_t(\dd x, \dd y) \\
& =  \int_{\mathbb{T}^{m}}  \po \Delta_x \varphi(x) \pi_t^\xi(x) + \po \int_{\mathbb{T}^{n-m}} \mathcal F_1(x,y) \frac{\pi_t(x,y)}{\pi_{t}^{\xi}(x)}\dd y + B_t(x)\pi_t^\xi(x)\pf \nabla_x \varphi(x) \pf \dd x  \\
& = \int_{\mathbb{T}^{m}} \po \Delta_x \varphi + (B_t-G_{t})\nabla_x \varphi\pf \pi_t^\xi\,.
\end{align*}
\end{proof}
Remark that in \cite[Proposition 2]{Alrachid}, the Helmoltz projection is done in $L^2(\pi_t^\xi)$, so that $\nabla\cdot((B_t-G_{t})\pi_t^\xi)=0$ and one ends up with the heat equation. Here, we get the heat equation in the ABF case ($B_{t}= G_{t}$) and, in the PABF case ($B_{t} = \mathsf{P}_{L^{2}(\lambda)}(G_t)$), an additional time-dependent divergence-free drift.
\begin{rem}\label{Rem-pi-xi}
Since the density $\pi_{t}^\xi$, as well as constants, satisfies the Fokker-Planck equation~\eqref{FP-Pi_xi} which preserves positivity, provided there exists $m_{0}^\xi>0$ such that $\pi_{0}^\xi \ge m_{0}^{\xi}$, one can show that $\pi_{t}^\xi \ge m_{0}^{\xi}$ for all $t\ge 0$ on the torus $\IT^m$. Note that if $\pi_{0}^\xi$ was to be zero at some points or not sufficiently smooth, the conditional mean $G_{0}$ given in (\ref{EqFokkerPlanckABF}) might not be well defined. 
\end{rem}
In view of Remark~\ref{Rem-pi-xi}, from now on, assume the following:
\begin{assu}\label{AssuCINI} The initial condition $\pi_0$ admits a smooth density with respect to the Lebesgue measure, such that $\pi_{0}^{\xi}$ is positive.
\end{assu}
As a consequence, the conditional means $G_{t}$ are well defined for all $t \ge 0$, along with the entropy $\mathcal{H}(\pi_{0} \,|\,\lambda)$, which is ensured to be finite. Furthermore, $\pi_{0}^{\xi}$ belongs to $L^{2}(\IT^{m})$. 

\medskip
 
Both the ABF and PABF algorithms are designed in order to ensure that all the values of the transition coordinate have been visited. In other words, the density of $\xi(X_{t},Y_{t})$ should converge to a flat histogram, namely the Lebesgue measure $\lambda$.  In the conservative case, this is known to hold in both the ABF case \cite{FreeEnergy} and the PABF case \cite{Alrachid}. We now extend the flat histogram property to the general --possibly non-conservative-- case.
\begin{prop}\label{PropHistoPlat}
For both the ABF and PABF algorithm, under Assumption~\ref{AssuCINI}, $\pi_t^\xi$ converges towards the Lebesgue measure as $t\rightarrow \infty$. More precisely, for all $t \geq 0$:
\[\mathcal H(\pi_t^\xi | \lambda)\ \leqslant \ e^{-8\beta^{-1}\pi^{2} t} \mathcal H(\pi_0^\xi | \lambda)\,.\]
\end{prop}

Furthermore, the entropic convergence of the density can be strengthened to an $L^\infty$ one, that will prove useful in the rest of the study:

\begin{prop}\label{PropHistoUniform}
For both the ABF and PABF algorithm, under Assumption~\ref{AssuCINI}, there exists $C>0$ such that for all initial distribution $\pi_{0}^{\xi} \in L^{2}(\IT^{m})$, for all $t \geq 1$:
$$\|\pi_t^\xi - 1\|_\infty \ \leqslant \ Ce^{-4\beta^{-1}\pi^2 t}\|\pi_{0}^{\xi}-1\|_{2} \,. $$
\end{prop}

 As detailed in \cite{LRS07,Alrachid}, in the conservative case $\mathcal F = -\nabla V$, $\pi_\infty=\mu_A$ is a stationary state of~\eqref{EqFokkerPlanckABF}. In the non-conservative case, the existence of such a stationary state may be unclear, and this issue will be treated in Theorem~\ref{ThmExistence} below, which will pe proved in Section \ref{subsec:proof-existence}. For now, let us consider the following assumption:

\begin{assu}\label{HypoGeneral}\
The interaction force $\mathcal{F}$ is in $\mathcal{C}^{1}(\IT^n, \IR^n)$, and we denote by $M>0$ a constant such that for all $y\in\mathbb T^{n-m}$, $x\mapsto \mathcal F_1(x,y)$ is $M$-Lipschitz.
\end{assu}

\begin{theo}\label{ThmExistence}
For both the ABF and  PABF algorithms, under Assumption~\ref{HypoGeneral}, there exists a couple of stationary measure and bias $\po \pi_{\infty}^{\mathcal{F}}, B_{\infty}^\mathcal{F}\pf$ to \eqref{EqFokkerPlanckABF}, such that $\pi_{\infty}^{\mathcal{F}} \in \mathcal{C}^{0}(\IT^n)$ is stricly positive. As a consequence,
\begin{itemize}
\item[(i)] $\pi_{\infty}^{\mathcal{F}}$ satisfies a log-Sobolev inequality for some constant $R>0$,
\item[(ii)] the conditional density $y\mapsto \pi_{\infty,x}^{\mathcal{F}}(y):=\pi_{\infty}^{\mathcal{F}}(x,y)/\pi_{\infty}^{\mathcal{F},\xi}(x)$ satisfies a log-Sobolev inequality for some constant $\rho>0$, for all $x \in \IT^{m}$.
\end{itemize}
\end{theo}

\begin{rem} Note that there is no reason whatsoever for $\pi_{\infty}^{\mathcal{F}}$ to be the same in both the ABF and PABF case. Nevertheless, as shown in Proposition~\ref{PropHistoPlat}, $\pi_{\infty}^{\xi} \equiv 1$ in all cases.
\end{rem}

\begin{rem} An important remark is that, at small temperatures (i.e. $\beta\gg 1$), the optimal log-Sobolev constant of a probability measure with density proportional to $\exp(\beta W)$ for some~$W$, roughly scales like $\exp(\beta d_{W})$ where $d_{W}$ is the so-called critical depth of $W$ \cite{Menz_2014} (the critical depth is the highest energy barrier to overcome in order to reach a global minimum of $W$). If the transition coordinate is well-chosen, the metastability in the orthogonal space should be small, meaning that for all $x\in\mathbb T^m$ the critical depth of $W(x,\cdot)$ should be small with respect to the critical depth of $W$. As a consequence, as a function of $\beta$, $\rho$ is expected to be much larger than the log-Sobolev constant of $\mu\varpropto e^{-\beta V}$, which is the convergence rate to equilibrium of the original (unbiased) dynamics \eqref{GeneralOverdampedLangevin}. 
\end{rem}
 
The following result deals with the robustness of the conservative equilibrium to non-conservative perturbations, and will be proved in Section~\ref{subsec:perturb-diffusion}.
\begin{prop}\label{PropPerturb}
For the PABF algorithm, under Assumption~\ref{AssuCINI} and Assumption~\ref{HypoGeneral}, for all $V \in \mathcal{C}^{2}(\IT^{n})$ and $p \geq 1$, there exists $ K_{V}>0$ and $K_{p}>0$ such that the following holds.
 Denote by $A$ the free energy associated to $V$ (see equation~\eqref{eq:free-energy} for the definition of $A$). For all  $\mathcal{F} \in \mathcal{C}^{1}(\mathbb{T}^{n})$ satisfying $\|\mathcal{F}+\nabla V \|_{\infty} \leq 1$, for all equilibrium measure  $\pi_{\infty}^{\mathcal{F}}$ of  \eqref{EqFokkerPlanckABF}, considering the corresponding bias $\nabla H_{\infty}^{\mathcal{F}}$, one has
 \[\| \nabla A - \nabla H_{\infty}^\mathcal F\|_{L^{p}(\IT^m)} \ \leq \ K_V K_p \|\mathcal{F}+\nabla V \|_{\infty}\,,\]
and, for  all $\psi\in L^\infty(\mathbb T^n)$, considering
 \[\hat I_{\psi} \ :=\ \displaystyle{ \frac{\int_{\mathbb T^n} \psi(x,y) e^{-\beta H_{\infty}^{\mathcal F}(x)} \pi_{\infty}^{\mathcal F}(x,y) \dd x \dd y }{\int_{\mathbb T^n}e^{-\beta H_{\infty}^{\mathcal F}(x)} \pi_{\infty}^{\mathcal F}(x,y)  \dd x \dd y}} \,,\]one has
 \[\left| \int_{\mathbb T^n} \psi \dd \mu -\hat I_{\psi} \right| \ \leq \ K_V  \left\| \psi  \right\|_\infty  \|\mathcal{F}+\nabla V \|_{\infty}\,.\]
\end{prop}

The first point of Proposition~\ref{PropPerturb} states that, if the error on the forces $\nabla V$ is small, then the bias of the free energy estimation is small.  The second point states that similarly, the bias on the computations of averages with respect to $\mu$ is small if the error on the forces is small. Indeed,  in practice, in order to compute averages with respect to the initial target law $\mu$  from the biased trajectory, two strategies are available: either standard importance sampling re-weighting, or estimation of the conditional expectations given $\xi(X,Y)=x$ and then average with respect to $\exp(-H_{\mathcal F}(x))$. In both cases, if $-\nabla V$ is replaced by $\mathcal F$ due to some numerical errors and the process converges in large time towards an equilibrium $\pi_\infty^{\mathcal F}$, then a quantity of the form $\int_{\mathbb T^n} \psi \dd \mu$ is approximated by an estimator that converges in large time towards the quantity $\hat I_\psi$ defined in Proposition~\ref{PropPerturb}.\\

Finally, we turn to the long-time convergence of the density $\pi_t$ on the whole space. 
 The first theorem concerns the classical, conservative case, whereas the second concerns the general case, where the force $\mathcal{F}$ can be non conservative. These will respectively be proved in sections~\ref{subsec:thm-conservative} and \ref{subsec:thm-non-conservative}.

\begin{theo}\label{ThmConservative}
Let us consider $(\pi_{t}, B_t)$ solution of ~\eqref{EqFokkerPlanckABF} for either the ABF or  PABF algorithm, under Assumption~\ref{AssuCINI} and Assumption~\ref{HypoGeneral}. Let us suppose moreover that $\mathcal F = -\nabla V$, with $V\in\mathcal C^2(\mathbb T^n)$. Then, there exists $K> 0$ such that, for all $\varepsilon>0$ and for all $t \ge 0$:
\begin{equation*}
\mathcal{H}\left( \pi_{t} |\mu_A \right) \leq K\left(1+\frac{1}{\varepsilon^{2}} \right) e^{-\left(\Lambda-\varepsilon\right) t},
\end{equation*}
with $\mu_{A}$ being given by \eqref{eq:mu-FE}, $\Lambda=\left(8\pi^{2} \wedge 2\rho\right)\beta^{-1} $ in the ABF case, $\Lambda = \left(4\pi^{2} \wedge  2\rho\right)\beta^{-1}$ in the PABF case, and $\rho$ is the log-Sobolev constant of the conditional density $y\mapsto \mu_{A,x}(y):=\mu_{A}(x,y)/\mu_{A}^{\xi}(x)$. Furthermore, \eqref{EqFokkerPlanckABF} consequently admits a unique stationary state: using the notations of Theorem~\ref{ThmExistence}, $\left( \pi_{\infty}^{-\nabla V}, B_{\infty}^{-\nabla V}\right) = \left(\mu_A, \nabla A\right)$.
\end{theo}
 
This extends \cite[Theorem 1]{LRS07}, which is restricted to the ABF algorithm with $m=1$. Besides, for the PABF algorithm, \cite[Theorem 1]{Alrachid} is a similar convergence result but for a variant of the algorithm where the classical Helmholtz projection in $L^{2}(\lambda)$ is replaced by the Helmholtz projection in the weighted space  $L^2(\pi_t^\xi)$. This variant is motivated in \cite{Alrachid} by some cancellations in the computations of the proofs.
 Nevertheless, as already noted in \cite{Alrachid}, the classical Helmholtz projection is used in practice. Theorem~\ref{ThmConservative} in the PABF case is thus a new result which fills a gap between the existing theoretical convergence results and the practical algorithm.
 
\begin{rem}
For $t>0$, applying Theorem~\ref{ThmConservative} with $\varepsilon = 1/t$ yields
$$\mathcal{H}\left( \pi_{t} |\mu_A \right) \leq Ke^{1}(1+t^{2}) e^{-\Lambda t}.$$
\end{rem}
 
The next results address the general --possibly non-conservative-- case, and as such are new.

\begin{theo}\label{ThmCVtempslong}
Let us consider $(\pi_{t}, B_t)$ solution of ~\eqref{EqFokkerPlanckABF} for either the ABF or  PABF algorithm, under Assumption~\ref{AssuCINI} and Assumption~\ref{HypoGeneral}. Let $\pi_{\infty}^{\mathcal{F}}$, $R,\rho$ be a stationary measure for ~\eqref{EqFokkerPlanckABF} and the two associated constants, as introduced in Theorem~\ref{ThmExistence}. Suppose moreover that $ M \beta<2\rho $. Then there exists $K\ge 0$ such that, for all $t\ge 0$:
\begin{equation*}
 \mathcal H \left(\pi_t | \pi_\infty^{\mathcal F} \right)\leq K\dst{e^{-\Lambda  t}},
 \end{equation*}
 with $\Lambda = 2 R (1-\frac{M\beta}{2\rho})\beta^{-1}$. As a consequence, the dynamics~\eqref{EqFokkerPlanckABF} admits a unique stationary state.
\end{theo}
Eventually, one has the following result, which will be proved in Section~\ref{subsec:corollary-cv}.
\begin{cor}\label{CorCvBias}
Under the settings of either Theorem~\ref{ThmConservative} or \ref{ThmCVtempslong}, there exists a unique stationary state $\po \pi_\infty^{\mathcal{F}}, B_\infty^{\mathcal{F}} \pf$ for the dynamics \eqref{EqFokkerPlanckABF}. Furthermore,
 there exists $K\geq 0$ such that for all  $t \ge 0$,
\begin{equation*}
\int_{\mathbb{T}^{m}} |B_t-B_\infty^{\mathcal{F}}|^2 \dd x \le  Ke^{-\Lambda  t}, 
\end{equation*}
where $\Lambda $  is given by either Theorem~\ref{ThmConservative} (where $\mathcal{F}=-\nabla V$) or \ref{ThmCVtempslong} (where $\mathcal{F}$ is general).
\end{cor}

\begin{rem}\label{Rem-CvMeasure} 
A direct consequence of the Csiz\`ar-Kullback inequality \eqref{CK} combined with either Theorem~\ref{ThmConservative} or Theorem \ref{ThmCVtempslong} is that for all $t \ge 0$
$$ \|\pi_{t}-\pi_{\infty}^{\mathcal{F}}\|_{TV} \leq \sqrt{2K} e^{-\frac{1}{2}\Lambda t},$$
where $K,\Lambda\ge 0$  are given by either Theorem~\ref{ThmConservative} (where $\mathcal{F}=-\nabla V$) or \ref{ThmCVtempslong} (where $\mathcal{F}$ is general).
\end{rem}

Theorem ~\ref{ThmCVtempslong} shows the exponential convergence to a unique stationary state for the ABF and PABF algorithms even for non-conservative forces. Notice that the rate of convergence obtained in Theorem~\ref{ThmConservative} for conservative forces is better than the rate of convergence in Theorem~\ref{ThmCVtempslong}. It would be interesting to further investigate the sharpness of these rates. \\

The rest of this paper is devoted to the proofs of the results stated in this section. From now on, and without loss of generality, we will assume that $\beta =1$. Note that the assumption of Theorem~\ref{ThmCVtempslong} now becomes $M<2\rho$. An adequate change of variable to then deduce the results for $\beta \neq 1$ is: $\tilde{t}=\beta^{-1} t$, $\tilde{\mathcal{F}}(x,y)=\beta \mathcal{F}(x,y)$, $\tilde{W}^{1}(x)=\beta W^{1}(x)$, $\tilde{W}^{2}(y)=\beta W^{2}(y)$, and $\tilde{\pi}_{t}(x,y)=\pi_{t}(x,y),$ for all $t\geq 0$ and for all $(x,y)\in \mathbb{T}^{n}$. 

\section{Law of the transition coordinate}
\label{sec:law-pi-xi}After proving in Section~\ref{subsec:dem-prop-histo-plat} the long-time entropic convergence of the density $\pi_{t}^{\xi}$ towards the Lebesgue measure $\lambda$, we prove in Section~\ref{subsec:dem-prop-histo-uniform} its long-time $L^{\infty}$-convergence, by relying on a Nash inequality on the $n$-dimensional torus and on the proof of \cite[Theorem 6.3.1]{BakryGentilLedoux}.

\subsection{Proof of Proposition \ref{PropHistoPlat}}
\label{subsec:dem-prop-histo-plat}

\begin{proof}
One has:
$$\frac{\dd}{\dd t} \mathcal H(\pi_t^\xi|\lambda)=\frac{\dd}{\dd t} \int_{\mathbb{T}^{m}} \pi_t^\xi \ln \pi_t^\xi.$$
Considering $\mathcal{L}_t\mu =  \nabla\cdot(\nabla \mu - (B_t-G_{t})\mu)$:
$$\frac{\dd }{\dd t} \displaystyle{\int_{\mathbb{T}^{m}}\pi^{\xi}_{t}}=\displaystyle{\int_{\mathbb{T}^{m}} \mathcal{L}_{t}\pi^{\xi}_{t}}=0.$$
One gets, using integration by parts,
\begin{align}
\frac{\dd }{\dd t} \mathcal H(\pi_t^\xi|\lambda)& =  \int_{\mathbb{T}^{m}} \ln \pi_t^\xi \mathcal{L}_t\pi_t^\xi + \int_{\mathbb{T}^{m}} \mathcal{L}_t \pi_t^\xi  \nonumber \\
& =  -\int_{\mathbb{T}^{m}} \frac{|\nabla \pi_t^\xi|^2}{\pi_t^\xi} + \int_{\mathbb{T}^{m}} (B_t-G_{t})\nabla \pi_t^\xi   \nonumber \\
& =  -\int_{\mathbb{T}^{m}} \frac{|\nabla \pi_t^\xi|^2}{\pi_t^\xi} \quad \text{(since $\nabla \cdot (B_{t}-G_{t})=0$)} \nonumber \\
&=-\displaystyle{\int_{\mathbb{T}^{m}} |\nabla \ln\left(\frac{\pi_t^\xi}{\lambda}\right)|^2\, \pi_t^\xi } \nonumber \\
&= - I(\pi^{\xi}_{t}|\lambda).\label{dEM}
\end{align}
Since the Lebesgue measure $\lambda$ satisfies a log-Sobolev inequality of constant $4\pi^{2}$ \cite[Proposition 5.7.5(ii)]{BakryGentilLedoux}, we have:
$$
\partial_t \mathcal H(\pi_t^\xi|\lambda) \leq -2(4\pi^{2}) \mathcal{H}(\pi^{\xi}_{t}|\lambda),
$$
 which concludes the proof of Proposition~\ref{PropHistoPlat}, denoting by $\pi_{\infty}^{\xi} \equiv \lambda$ the long-time limit of $\pi_{t}^{\xi}$.
\end{proof}

\subsection{Proof of Proposition \ref{PropHistoUniform} }
\label{subsec:dem-prop-histo-uniform}

We first state a Nash inequality on the $n$-dimensional torus.
\begin{lem}\label{LemNashTorusn}
For all $n \in \IN^{*}$, there exists $a=a(n)>0$ such that for all function $u \in H^{1}(\IT^{n})$:
\begin{equation}
\|u\|_{2}^{2} \leq 2\|u\|_{1}^{2}+a\|\nabla u\|_{2}^{\frac{2n}{n+2}} \|u\|_{1}^{\frac{4}{n+2}}.\label{NashTorusn}
\end{equation}
\end{lem}
\begin{proof}
Let us recall that $\IT^{n}=\IR^{n}/\IZ^{n}$. 
 We consider $L^{2}(\mathbb{T}^{n})$ equipped with the inner product $\langle u,v\rangle :=\int_{\mathbb{T}^{n}} u(x)\bar{v}(x)\, \dd x $. The sequence $\{ e^{2\pi i k x}\}_{k \in \IZ^{n}}$ is an orthonormal basis of $L^{2}(\IT^{n})$. Now given a function $u \in L^{2}(\IT^{n})$ and its Fourier coefficients
$$c_{k}=\dst{\int_{\IT^{n}} u(x) e^{-2\pi ikx} \, \dd x}, \qquad \forall k \in \IZ^n,$$
denoting by $k=(k_{1}, \ldots , k_{n})$ a vector in $\IZ^{n}$, and $|k|=\sqrt{\sum_{j=1}^{n}|k_{j}|^{2}}$, the Parseval identity yields
$$\|u\|_{2}^{2} = \sum_{k \in \IZ^{n}} |c_{k}|^{2} \,,\qquad \|\nabla u\|_{2}^{2} = \sum_{k \in \IZ^{n}} |k|^{2}\,|c_{k}|^{2}.$$
Let $\rho >0$ to be fixed later on. One has, considering $\|k\|_{\infty}=\max\limits_{j\in \llbracket 1, n\rrbracket}\{|k_{j}|\}$:
\begin{align*}
\|u\|_{2}^{2} &=  \sum_{k \in \IZ^{n}} |c_{k}|^{2} =\sum_{\|k\|_{\infty} \leq \rho} |c_{k}|^{2} + \sum_{\|k\|_{\infty} > \rho} |c_{k}|^{2}\\
&\leq  \sum_{\|k\|_{\infty} \leq \rho }|c_{k}|^{2}+\frac{1}{\rho^{2}}\sum_{\|k\|_{\infty} > \rho}\|k\|_{\infty}^{2} |c_{k}|^{2}\\
&\leq  \sum_{\|k\|_{\infty} \leq \rho }|c_{k}|^{2}+\frac{1}{\rho^{2}}\sum_{\|k\|_{\infty} > \rho}|k|^{2} |c_{k}|^{2}\\
&\leq  \sum_{\|k\|_{\infty} \leq \rho }|c_{k}|^{2}+\frac{1}{\rho^{2}}\|\nabla u \|_{2}^{2}.
\end{align*}
And:
\begin{align*}
\sum_{\|k\|_{\infty} \leq \rho }|c_{k}|^{2} &= \sum_{\|k\|_{\infty} \leq \rho } \left| \dst{\int_{\mathbb{T}^{n}} u(x) e^{-2\pi ikx} \, \dd x} \right|^{2} \ \leq \    \|u\|_{1}^{2} \sum_{\|k\|_{\infty} \leq \rho }1 \ \leq\   (2\rho+1)^{n} \|u\|_{1}^{2}.
\end{align*}
Consequently:
\begin{equation}
\|u\|_{2}^{2} \leq 3^{n}(\rho \vee 1)^{n}  \|u\|_{1}^{2}+\frac{1}{\rho^{2}}\|\nabla u \|_{2}^{2}.\label{NashFourierStep1}
\end{equation}
We now distinguish between two cases:
\begin{itemize}
\item[(i)] If $3^{n}\| u\|_{1}^{2} \leq  \|\nabla u \|_{2}^{2}$, by choosing
$$\rho= 3^{-\frac{n}{n+2}} \frac{\|\nabla u\|_{2}^{\frac{2}{n+2}}}{\|u\|_{1}^{\frac{2}{n+2}}} \geq 1,$$  inequality \eqref{NashFourierStep1} 
 yields:
\begin{align}
\|u\|_{2}^{2}&\leq  3^{n}3^{-\frac{n^{2}}{n+2}} \|\nabla u\|_{2}^{\frac{2n}{n+2}}\|u\|_{1}^{\frac{4}{n+2}} + 3^{\frac{2n}{n+2}}\|\nabla u\|_{2}^{\frac{2n}{n+2}}\|u\|_{1}^{\frac{4}{n+2}}\nonumber\\
&= 2\cdot 3^{\frac{2n}{n+2}}\|\nabla u\|_{2}^{\frac{2n}{n+2}}\|u\|_{1}^{\frac{4}{n+2}}.\label{NashFourierCaseA}
\end{align}
\item[(ii)] If $3^{n}\| u\|_{1}^{2} \geq  \|\nabla u \|_{2}^{2}$, one wishes to rely on the Poincaré-Wirtinger inequality on the torus~$\IT^n$. The optimal Poincaré constant in $H^{1}_{0}(\IT^{n})$ being equal to $\lambda_{1}^{-1}$, where $\lambda_{1}=4\pi^2$ is the first non trivial eigenvalue of the negative Laplacian $-\Delta$, one can consider the following Poincaré-Wirtinger inequality:
\begin{equation}
\|u-\bar{u}\|_{2}^{2} \leq \frac{1}{4\pi^{2}}\|\nabla u\|_{2}^{2}, \qquad  \forall u \in H^{1}(\IT^{n}),\label{PWnTorus}
\end{equation}
where $\bar{u}=\dst{\int_{\IT^{n}} u(x) \, \dd x}$. One consequently gets:
\begin{align}
\|u\|_{2}^{2} &\leq  2\bar{u}^{2}+2\|u-\bar{u}\|_{2}^{2} \nonumber\\
&\leq  2\|u\|_{1}^{2} +2\frac{1}{4\pi^{2}} \|\nabla u\|_{2}^{2} \nonumber \\
&=  2\|u\|_{1}^{2} +\frac{1}{2\pi^{2}} \|\nabla u\|_{2}^{\frac{2n}{n+2}}\|\nabla u \|_{2}^{\frac{4}{n+2}} \nonumber \\
&\leq  2\|u\|_{1}^{2} +\frac{1}{2\pi^{2}}3^{\frac{2n}{n+2}} \|\nabla u\|_{2}^{\frac{2n}{n+2}}\| u \|_{1}^{\frac{4}{n+2}}.\label{NashFourierCaseB}
\end{align}
\end{itemize}
Combining \eqref{NashFourierCaseA} and \eqref{NashFourierCaseB}, one obtains:
$$
\|u\|_{2}^{2} \leq 2\|u\|_{1}^{2}+  3^{\frac{2n}{n+2}} \po\frac{1}{2\pi^{2}} \vee 2\pf \|\nabla u \|_{2}^{\frac{2n}{n+2}}\|\nabla u \|_{1}^{\frac{4}{n+2}}
$$
which yields \eqref{NashTorusn}, with $a=2\cdot 3^{\frac{2n}{n+2}} $.
\end{proof}

We are now in position to prove Proposition~\ref{PropHistoUniform}.
\begin{proof}[Proof of Proposition~\ref{PropHistoUniform}]
We will rely on the idea of the proof of \cite[Theorem 6.3.1]{BakryGentilLedoux}. Let us start with two preliminary results. Let $\varphi \in \mathcal{C}^{\infty}(\mathbb{T}^{m})$ be a test function and consider:
$$\forall z \in \mathbb{T}^{m}, \qquad \varphi_{t}(z)=\mathbb{E}_{z}[\varphi(Z_{t}) ]=\mathbb{E}[\varphi(Z_{t})\, |\, Z_{0}=z],$$
where $(Z_{t})_{t\geq 0}$ satisfies the following dynamics:
$$
\dd Z_{t}= \left(B_{t}-G_{t}\right)(Z_{t})\dd t +\sqrt{2} \dd W_{t},$$
where $(W_{t})_{t\geq 0}$ is a $n$-dimensional Brownian motion and $\nabla \cdot \left(B_{t}-G_{t}\right)=0$. Let $\nu_{Z}$ be the invariant measure of this dynamics, $\mathcal{L}=\left(B_{t}-G_{t}\right) \cdot \nabla + \Delta$ its infinitesimal generator, and $\mathcal{L}^{*}=-\nabla \cdot \left(B_{t}-G_{t}\right) +\Delta$ its adjoint in $L^{2}(\nu_{Z})$.
Using Itô calculus, $\varphi_{t}$ satisfies:
\begin{equation}
\varphi_{0}=\varphi, \quad \partial_{t}\varphi_{t}=\Delta \varphi_{t} + \left(B_{t}-G_{t}\right)\cdot\nabla  \varphi_{t}  \label{EqPhi}
\end{equation}
which is equivalent to
\begin{equation*}
\varphi_{0}=\varphi, \quad \partial_{t}\varphi_{t}=\Delta \varphi_{t} +\nabla \cdot \left((B_{t}-G_{t})\varphi_{t}\right).
\end{equation*}
Given the result of Lemma~\ref{Lemme-loi-pi^xi}, $\pi_{t}^{\xi}-1$ satisfies:
\begin{equation}
\partial_{t} \left(\pi_{t}^{\xi}-1\right)= \Delta \left(\pi_{t}^{\xi}-1\right) - \nabla \cdot \left((B_{t}-G_{t}) (\pi_{t}^{\xi}-1)\right). \label{EqXimoinsUn}
\end{equation}

For a fixed $t>0$, one as, for all $0 \leq s \leq t$:
\begin{align*}
\frac{\dd}{\dd s} \int_{\IT^{n}} \varphi_{t-s} \left(\pi_{s}^{\xi}-1\right) &= -\int_{\IT^{n}} \mathcal{L}\varphi_{t-s} \left(\pi_{s}^{\xi}-1\right) +\int_{\IT^{n}} \varphi_{t-s} \mathcal{L}^{*}\left(\pi_{s}^{\xi}-1\right)=0.
\end{align*}
Integrating between $s=0$ and $s=t$ yields
\begin{equation}
\dst{\int_{\mathbb{T}^{m}} \varphi_{t}(\pi_{0}^{\xi}-1)}=\dst{\int_{\mathbb{T}^{m}} \varphi(\pi_{t}^{\xi}-1)}, \qquad \forall t \geq 0. \label{eq:equality-phi-t-phi-0}
\end{equation}

Second, for all $t\geq 0$, 
\begin{equation}
\|\varphi_{t}\|_{1} \leq \|\varphi\|_{1}.
\label{eq:ineq-phi-t-phi-0}
\end{equation}
Indeed, one has on the torus $\IT^{m}$:
$$\|\varphi_{t}\|_{1} \leq \dst{\int_{\IT^{m}} \psi(t,z) \, \dd z}$$
where, for all $t \geq 0$ and $z\in \IT^{n}$, $\psi(t,z)=\IE[\, |\varphi(Z_{t})| \, |\, Z_{0}=z\,]\geq 0$ satisfies \eqref{EqPhi} with initial condition $\psi(0,.)=|\varphi(.)| \geq 0$ on $\IT^{m}$. Integrating by parts and using that $\nabla\cdot(B_t-G_t)=0$ one can check that $\frac{\dd }{\dd t}\int_{\IT^{m}} \psi =0$, so that:
$$\dst{\int_{\IT^{m}} \psi(t,z) \, \dd z }= \dst{\int_{\IT^{m}} \psi(0,z) \, \dd z}= \|\varphi\|_{1}, \quad \forall t \geq 0,$$
hence the result.\\

\noindent\textbf{Step 1:} Now let us show that there exists $\mathcal{C}>0$ such that, for all $t > 0$, $$\|\varphi_{t}\|_{2}^{2} \leq \left(\mathcal{C}t^{-\frac{m}{2}}+2\right) \|\varphi\|_{1}^{2}.$$ To do so, consider for all $t\geq 0$, $$\Lambda(t)=\dst{\int_{\mathbb{T}^{m}} |\varphi_{t}|^{2}}.$$
Since $\nabla \cdot\left(B_{t}-G_{t}\right)=0$ one can show from \eqref{EqPhi} that:
$$\Lambda'(t)=-2\dst{\int_{\mathbb{T}^{m}} |\nabla \varphi_{t}|^{2}}.$$
Knowing that $\|\varphi_{t}\|_{1} \leq \|\varphi\|_{1}$ for all time $t\geq 0$, we use the inequality \eqref{NashTorusn} given by Lemma~\ref{LemNashTorusn} to obtain:
$$\Lambda(t) \leq 2\|\varphi\|_{1}^{2}+a\left[-\frac{1}{2}\Lambda'(t)\right]^{\frac{m}{m+2}}\|\varphi\|_{1}^{\frac{4}{m+2}}.$$
Consider for all $t\geq 0$, $g(t)=\Lambda(t)-2\alpha$, where $\alpha=\|\varphi\|_{1}^{2}$. By construction, $g$ is decreasing on $\IR^{+}$. We distinguish between three cases:
\begin{itemize}
\item[(i)] Assume that $g(0) \leq 0$. In this case, $g(t) \leq 0$ for all $t \geq 0$ and, for all $t \geq 0$:
$$\|\varphi_{t}\|_{2}^{2} \leq 2 \|\varphi\|_{1}^{2}.$$
\item[(ii)] Assume that $g(t) >0$ for all $t\geq 0$. Then:
\begin{align*}
 g(t) \leq a\alpha^{\frac{2}{m+2}} \left[-\frac{1}{2}g'(t)\right]^{\frac{m}{m+2}}
&\Leftrightarrow  g(t)^{\frac{m+2}{m}} \leq -\frac{1}{2}a^{\frac{m+2}{m}} \alpha^{\frac{2}{m}}g'(t)\\
&\Leftrightarrow  g'(t) \leq -2\cdot a^{-\frac{m+2}{m}} \alpha^{-\frac{2}{m}}g(t)^{\frac{m+2}{m}} \\
&\Leftrightarrow   g'(t)g^{-\frac{m+2}{m}} \leq -2\cdot a^{-\frac{m+2}{m}}\alpha^{-\frac{2}{m}}  \\
&\Leftrightarrow  -\frac{m}{2}\frac{\dd }{\dd t } \left(g(t)^{-\frac{2}{m}} \right) \leq -2\cdot a^{-\frac{m+2}{m}}\alpha^{-\frac{2}{m}}\\
&\Leftrightarrow  \frac{\dd }{\dd t } \left(g(t)^{-\frac{2}{m}} \right) \geq \frac{4}{m}\cdot a^{-\frac{m+2}{m}} \alpha^{-\frac{2}{m}}.
\end{align*}
Integrating between $0$ and $t$ yields:
\begin{align*}
g(t)^{-\frac{2}{m}} \geq g(0)^{-\frac{2}{m}} +\frac{4}{m}\cdot a^{-\frac{m+2}{m}}\alpha^{-\frac{2}{m}}t &\Leftrightarrow  g(t)^{\frac{2}{m}} \leq \dst{\frac{1}{g(0)^{-\frac{2}{m}} +\frac{4}{m}\cdot a^{-\frac{m+2}{m}}\alpha^{-\frac{2}{m}}t }}  \\
&\Leftrightarrow  g(t)^{\frac{1}{m}} \leq \dst{\frac{1}{\left(g(0)^{-\frac{2}{m}} +\frac{4}{m}\cdot a^{-\frac{m+2}{m}} \alpha^{-\frac{2}{m}}t \right)^{\frac{1}{2}}}}  \\
&\Rightarrow  g(t)^{\frac{1}{m}} \leq \dst{\frac{1}{\left(\frac{4}{m}\cdot a^{-\frac{m+2}{m}}  \alpha^{-\frac{2}{m}}t \right)^{\frac{1}{2}}}}\\
&\Leftrightarrow  g(t) \leq 2^{-m}a^{\frac{m+2}{2}}m^{\frac{m}{2}} \alpha t^{-\frac{m}{2}}.
\end{align*}
Eventually for all $t \geq 0$:
$$ \|\varphi_{t}\|_{2}^{2} \leq \left(\mathcal{C}t^{-\frac{m}{2}}+2\right) \|\varphi\|_{1}^{2}$$
with $\mathcal{C}=2^{-m}a^{\frac{m+2}{2}}m^{\frac{m}{2}}>0$.

\item[(iii)] Assume that $g(0)> 0$ and let us assume that $t^{*}>0$ is the smallest time $t$ such that $g(t^{*})\leq 0$. In this case, using the above reasonings, one obtains:
\begin{itemize}
\item[a)]For all $t\geq t^{*}, \, g(t) \leq g(t^{*})<0$ and thus
$$\|\varphi_{t}\|_{2}^{2} \leq 2 \|\varphi\|_{1}^{2}.$$
\item[b)]For all $t\in [0,t^{*}[, \, g(t)> 0$ and thus:

$$ \|\varphi_{t}\|_{2}^{2} \leq \left(\mathcal{C}t^{-\frac{m}{2}}+2\right)  \|\varphi\|_{1}^{2}.$$
\end{itemize}
\end{itemize}
Hence, for all $t \geq 0$, $\|\varphi_{t}\|_{2}^{2} \leq \left(\mathcal{C}t^{-\frac{m}{2}}+2\right) \|\varphi\|_{1}^{2}$.\\

\noindent\textbf{Step 2:} Now, for all $t \geq 0$, equation~\eqref{eq:equality-phi-t-phi-0} yields:
$$\left| \dst{\int_{\IT^{m}} \varphi \left(\pi_{t}^{\xi}-1 \right) } \right|^{2} = \left| \dst{\int_{\IT^{m}} \varphi_t \left(\pi_{0}^{\xi}-1 \right) } \right|^{2}.$$
Hence, for all $t \geq 0$:
\begin{align*}
\left| \dst{\int_{\IT^{m}} \varphi \left(\pi_{t}^{\xi}-1 \right) } \right|^{2} &\leq  \|\varphi_t\|_{2}^{2} \|\pi_{0}^{\xi}-1\|_{2}^{2}\\
&\leq  \left( Ct^{-\frac{m}{2}}+2 \right) \|\varphi\|_{1}^{2} \|\pi_{0}^{\xi}-1\|_{2}^{2} \quad \text{(using Inequality~\eqref{eq:ineq-phi-t-phi-0})}
\end{align*}
Since this is true for any function $\varphi \in L^{1}(\IT^{m})$, by duality, for all $t \geq 0$:
 \begin{equation}
\|\pi_{t}^{\xi}-1\|_{\infty} \leq \sqrt{\left(\mathcal{C}t^{-\frac{m}{2}}+2\right)}\|\pi_{0}^{\xi}-1\|_{2}.\label{BoundNormProofPropHistoUniform}
\end{equation}
\noindent\textbf{Step 3:} Considering the equation satisfied by $\pi_{t}^\xi$ given in Lemma~\ref{Lemme-loi-pi^xi}, with initial condition $\pi_{s}^{\xi}$ with $s \geq 0$, and using inequality \eqref{BoundNormProofPropHistoUniform}  over the time interval $[s,s+1]$, there exists $\mathcal{K}=\mathcal{K}(m)>0$ such that:
$$\|\pi_{s+1}^{\xi}-1\|_{\infty} \leq  \mathcal{K}\|\pi_{s}^{\xi}-1\|_{2}.$$
Denote by $H^{1}_{0}(\IT^{n})$ the closure of the space $\mathcal{C}^{\infty}_{0}(\IT{^n})$ of indefinitely differentiable functions with compact support, with respect to the Sobolev norm $\|\cdot \|_{H^{1}}$. Using the same reasoning as in the proof of Proposition~\ref{PropHistoPlat}, since $\int_{\IT^{m}}(\pi_{t}^{\xi}-1)=0$, $(\pi_{t}^{\xi}-1)$ belongs in $H^{1}_{0}(\IT^{n})$, and, using equation \eqref{EqXimoinsUn} and the Poincaré-Wirtinger inequality \eqref{PWnTorus}, one has:
\begin{equation}\label{eq:pixiL2}
\|\pi_{t}^{\xi}-1\|_{2} \leq \|\pi_{0}^{\xi}-1\|_{2}e^{-4\pi^{2}t}, \quad \forall t \geq 0.
\end{equation}
Eventually, for all $t \geq 1$:
$$\|\pi_{t}^{\xi}-1\|_{\infty} \leq \mathcal{K}\|\pi_{t-1}^{\xi}-1\|_{2}\leq \mathcal{K}e^{-4\pi^{2}(t-1)}\|\pi_{0}^{\xi}-1\|_{2},$$
which concludes the proof with $\mathcal{C}=\mathcal{K}e^{4\pi^{2}}$.
\end{proof}
\begin{rem}
Note that one could use the maximum principle for times $t \in [0,1]$ in order to replace the right-hand term $\|\pi_{0}^{\xi}-1\|_{2}$ by the $L^\infty$-norm $\|\pi_{0}^{\xi}-1\|_{\infty}$. Indeed, since by~Assumption~\ref{AssuCINI}, $\pi_{0}^{\xi}$ is continuous on $\IT^m$, one has a uniform bound on $\pi_{0}^{\xi}-1$. Nevertheless, considering an $L^2$-bound highlights the fact that the uniform bound at time $0$ is not essential to the proof of Proposition~\ref{PropHistoUniform}\nv{, which could be useful for possible generalizations to non-bounded state space cases.}
\end{rem}

\section{Existence of a stationary measure }
\label{sec:existence-equilibrium}
In Section~\ref{subsec:equilibrium-estimates-diffusion} we state and prove preliminary estimates on the invariant probability measures of homogeneous diffusions. We then proceed in Section~\ref{subsec:proof-existence} to prove Theorem~\ref{ThmExistence}, which gives the existence of a stationary state to~\eqref{EqFokkerPlanckABF} in the general case, where the force $\mathcal{F}$ can be non-conservative. Eventually, one can find in Section~\ref{subsec:perturb-diffusion} the proof of Proposition~\ref{PropPerturb} where one establishes bounds on the bias of the free energy estimation and on the bias on the computations of averages with respect to $\mu$.

\subsection{Preliminary estimates for homogeneous  diffusions}
\label{subsec:equilibrium-estimates-diffusion}

The next section is concerned with the sensitivity of the equilibrium measure of a diffusion with respect to its drift, when this drift is in $L^p$ for some $p$. Consider the following process on $\mathbb{T}^{n}$, with $n\geq 1$:
 \begin{equation}
\dd X_{t}=a(X_{t})\dd t +\sqrt{2} \dd W_{t} \label{EqA}
\end{equation}
with $(W_{t})_{t\geq 0}$  a classical $n$-dimensional Brownian motion on the torus $\IT^{n}$ and $a\in L^p(\mathbb{T}^n,\IR^n)$ for $p\geqslant 2$ with  $p>n$. We refer to \cite{Krylov} for a probabilist study of this SDE (existence, strong Markov and Feller properties, existence and Hölder continuity of the transition kernel, etc.). In the following we take a PDE point of view, namely we are interested in the existence, uniqueness and properties of a solution $\nu$ in $H^1(\mathbb{T}^n)$ such that $\int_{\IT^n} \nu(x) \, \dd x=1$ of the following equation:
\begin{equation}\label{eq:equilibreFPa}
\forall \varphi \in H^1(\mathbb{T}^n),\qquad \int_{\mathbb{T}^n} \po a(z)\cdot \nabla \varphi(z) \nu(z) - \nabla \varphi(z) \cdot \nabla \nu(z) \pf \dd z \ = \ 0\,.
\end{equation}
This implies in particular that $\int_{\mathbb T^n} (\mathcal L \varphi) \nu = 0$ for all $\varphi\in\mathcal C^2(\mathbb T^n)$ with  $\mathcal L$ being the generator of~\eqref{EqA}.
\begin{rem}
Note that the Sobolev embedding $H^1 \hookrightarrow L^q$ for some $q$ such that $\frac{1}{q}>\frac{1}{2}-\frac{1}{n}$ and the assumption that $p>n$ ensure that the integrals in~\eqref{eq:equilibreFPa} are well defined for all $\nu, \varphi$ in $H^{1}(\IT^n)$ and all $a \in L^{p}(\IT^n,\IR^n)$.
\end{rem}

\begin{prop}\label{Prop:driftLp}
Let $M>0$ and $p>n \geqslant 1$ with $p\geqslant 2$. There exists $C>0$ which depends solely on $M, p$ and $n$, such that the following holds. For all $a\in L^{p}(\IT^{n},\IR^n)$ such that $\|a\|_{L^p(\mathbb T^n)}\leqslant M$, there exists a unique probability density $\nu_a \in H^1(\mathbb T^n)$ that solves \eqref{eq:equilibreFPa}, and which is  such that
\[\|\nu_a\|_\infty + \|1/\nu_a\|_\infty + \|\nu_a\|_{H^1(\mathbb T^n)} \ \leqslant \ C\,.\]
 Moreover, if $\nu_b$ is the solution of \eqref{eq:equilibreFPa} with $a$ replaced by $b\in  L^{p}(\IT^{n},\IR^n)$ with $\|b\|_{L^p(\mathbb T^n)}\leqslant M$, then
 \[\|\nu_{a}-\nu_{b}\|_{L^2(\IT^{n})} \leq C\|a-b\|_{L^{2}(\IT^{n})}. \]
\end{prop}

\begin{rem}\label{Rem-W1p} 
In the case of a gradient drift $a=-\nabla \mathcal{A}$, the invariant measure $\nu_{a}$ is explicit: for all $z \in \IT^n$,
$$\nu_{a}(z)= \frac{1}{Z_{\mathcal{A}}} e^{-\mathcal{A}(z)}, \qquad Z_{\mathcal{A}}=\int_{\IT^n} e^{-\mathcal{A}(z)} \, \dd z,$$
and the $\mathcal L^\infty$-bound of Proposition~\ref{Prop:driftLp} amounts to the continuous injection given by Morrey's inequality ~\cite[Theorem IX.12]{Brezis}, 
\begin{equation}
W^{1,p}(\IT^{n})\hookrightarrow L^\infty(\IT^n), \qquad \forall p>n.\label{Morrey}
\end{equation}
Indeed, if $\mathcal{A} \in W^{1,p}(\IT^{n})$, then $\mathcal{A} \in L^{\infty}(\IT^{n})$ and $\nu_{a}$ is bounded from above and below (and conversely if $\nu_a$ is bounded above and below then $\mathcal A$ is bounded). In particular, since this injection is false for $p\leqslant n$, we see that the condition $p>n$ is necessary in  Proposition~\ref{Prop:driftLp} .
\end{rem}

\begin{proof}
\noindent\textbf{Step 1:} 
First assume that  $a \in \mathcal{C}^{\infty}(\IT^n, \IR^n)$. By \cite[Theorem 5.11]{Dynkin65}, there exists a Markov process $(\tilde{X}_t)_{t \geq 0}$ on $\IR^n$ whose transition probability density is given by the fundamental solution of the equation $\partial_t  \tilde{f}_t = -\Div \left(a \tilde{f}_t - \nabla \tilde{f}_t\right)$, where $a$ is seen as a $1$-periodic function on $\IR^n$. Note that by \cite[Theorem 0.5 and Condition 0.24.A1]{DynkinVol2}, the density $\tilde{f}_t$ is strictly positive and depends continuously on the initial condition. Moreover, \cite[Theorems 11.4 and 11.5]{Dynkin65} yield that $(\tilde{X}_t)_{t \geq 0}$ solves the stochastic differential equation~\eqref{EqA} on $\IR^n$. Now, consider $(X_t)_{t\geqslant 0}$ the image of $(\tilde X_t)_{t\geqslant 0}$ by the canonical projection from $\IR^n$ to $\IT^n$. Since $a$ is periodic,  $(X_t)_{t\geqslant 0}$ solves \eqref{EqA} as an equation on $\IT^n$, and thus, using Itô's formula, it is a Markov process (the proof is the same as \cite[Theorems  11.5]{Dynkin65} in $\IR^n$). Denote by $(P_t)_{t\geqslant 0}$ the associated Markov semigroup on $L^\infty(\IT^n)$. The positivity and continuity in the initial condition of $\tilde{f}_t$ implies that, for all $t>0$, there exists $r_t>0$ such that for all $x\in\IT^n$ and all Borel set $A$ of $\IT^n$, $\mathbb P_x(X_t\in A) \geqslant r_t \lambda(A)$, namely the process satisfies a uniform Doeblin condition. In particular, for a fixed $t>0$, the Markov chain with transition operator $P_t$ is recurrent and irreducible and thus, by \cite[Theorem 10.0.1]{MeynTweedie93}, it admits a unique invariant measure $\nu_a$. Now, for $s\geqslant 0$,  $(\nu_a P_s)P_t = (\nu_a P_t)P_s= \nu_a P_s$, which means that $\nu_a P_s$ is an invariant measure for $P_t$. Hence by uniqueness, $\nu_a P_s = \nu_a$ for all $s\geqslant 0$. In other words, $\nu_a$ is  the unique invariant measure for the semigroup $(P_t)_{t \geq 0}$.

Now, let $ \varphi \in \mathcal{C}^{2}(\IT^n)$.  Denoting by $\mathcal{L}= a\cdot \nabla + \Delta$  the infinitesimal generator of \eqref{EqA} and using Itô's formula, one gets for all $t \geq 0$
\[0 \ = \ \nu_a \left(P_{t}(\varphi)-\varphi\right) \ = \   \int_0^t \nu_a P_s \mathcal L\varphi \dd s \ = \ t \,\nu_a(\mathcal L\varphi)\,.\]
In other words, $\nu_a$ is a solution of the weak equation
\begin{equation}
 \forall \varphi \in \mathcal{C}^{2}(\IT^n), \qquad \nu_a(\mathcal{L}\varphi) = 0.\label{eq:FP-Kolmogorov-C2}
\end{equation}
By elliptic regularity (e.g.  \cite{Hormander1967} applied to $\nu_a$ seen as a periodic measure on $\IR^n$), 
  $\nu_a$ has then a $\mathcal{C}^{\infty}$ density (that we still denote by $\nu_a$) and, integrating by parts, we can write \eqref{eq:FP-Kolmogorov-C2} as
\begin{equation*}
  \int_{\IT^n} \po a(z)\cdot \nabla \varphi(z) \nu_a(z) - \nabla \varphi(z) \cdot \nabla \nu_a(z) \pf \dd z \ = \ 0
\end{equation*}
for all $\varphi \in \mathcal{C}^{2}(\IT^n)$ and thus, by density, for all $\varphi \in H^1(\IT^n)$. This is \eqref{eq:equilibreFPa}.

Define $\tilde{\nu}_{a}$  on $\IR^n$ by $\tilde{\nu}_{a}(x+k)=\nu_{a}(x)$ for all $k\in\IZ^n$ and $x\in \IT^n$ (seen as $[0,1]^n$). It is such that
\begin{equation*}
\forall \varphi \in H^1(\IR^n),\qquad \int_{\IR^n} \po a(z)\cdot \nabla \varphi(z) \tilde{\nu}_{a}(z) - \nabla \varphi(z) \cdot \nabla \tilde{\nu}_{a}(z) \pf \dd z \ = \ 0\,,
\end{equation*}
 where, again, $a$ is seen as a $1$-periodic function. Since $p>n$, using the notations of \cite{BogachevKrylov} and applying the Harnack inequality \cite[Corollary 1.7.2]{BogachevKrylov}, with the operator $L_{I_{n},a,0}$ ($I_n$ being the identity matrix of size $n$) and the domain $\Omega = [-1,2]^{n}$ which stricly contains $[0,1]^n$, we get that there exists $C_{1}>0$ depending only on $M$, $p$ and $n$ such that:
$$\sup_{z \in [0,1]^n} \tilde{\nu}_{a}(z) \le C_{1} \inf_{z \in [0,1]^{n}} \tilde{\nu}_{a}(z).$$
Using that $\dst{\int_{\IT^n}\nu_a }= 1$, this implies that
\begin{equation}\label{eq:Harnack}
1 \leqslant \ \sup_{z \in \IT^n} \nu_a  \leqslant C_1  \inf_{z \in \IT^n} \nu_a \leqslant C_1 \,.
\end{equation}
Taking $\varphi = \nu_{a}$ in  \eqref{eq:equilibreFPa} and using the Cauchy-Schwarz inequality yields
\[\int_{\mathbb T^n} |\nabla \nu_a|^2 \ = \ \int_{\mathbb T^n} a \cdot \nabla \nu_a \nu_a \ \leqslant \  \|\nu_a\|_\infty \|a\|_{L^2(\mathbb T^n)} \|\nabla \nu_a\|_{L^2(\mathbb T^n)}, \]
hence $\|\nabla \nu_a\|_{L^2(\mathbb T^n)} \leqslant MC_1$. Consequently, using the Poincaré-Wirtinger inequality~\eqref{PWnTorus}, $\|\nu_a\|_{H^1(\mathbb T^n)}\leqslant C_2$ for some $C_2>0$ that depends only on $M,p,n$.

\medskip

\noindent\textbf{Step 2:} Now we consider $a\in L^{p}(\IT^{n},\IR^n)$, with $\|a\|_{L^{p}(\IT^n)} \le M$, and proceed to prove the existence of a solution $\nu_{a}$ to equation~\eqref{eq:equilibreFPa}. Let $(a_k)_{k\in\IN}$ be a sequence of $\mathcal C^\infty$ functions that converges to $a$ in $L^p(\mathbb T^n)$ and such that $\|a_k\|_{L^p(\mathbb T^n)} \leqslant M$ for all $k\in\IN$. Let $(\nu_{a_k})_{k\in\IN}$ be the associated solutions of \eqref{eq:equilibreFPa} given in Step 1. From Step 1,  $(\nu_{a_k})_{k\in\IN}$ is bounded in $H^1(\IT^n)$, and thus we can consider a subsequence that converges weakly in $H^1$ and  strongly in $L^2$ 
 to some $\nu_a \in  H^1(\mathbb T^n)$. The weak convergence in $H^1$ implies that $\nu_a$ solves \eqref{eq:equilibreFPa} and $\|\nu_a\|_{H^1(\mathbb T^n)} \leqslant C_2$. The $L^2$-convergence 
  implies that $\nu_a$ is a probability density.

\medskip

\noindent\textbf{Step 3:} Let us now consider any solution and establish bounds similar to the previous step and a Poincaré inequality. For $a\in L^{p}(\IT^{n},\IR^n)$, let $\nu_a\in H^1(\IT^n)$ be any probability density solution of \eqref{eq:equilibreFPa}. Using again \cite[Corollary 1.7.2]{BogachevKrylov} and the fact that the mass of $\nu_a$ is $1$, we get  that $1/C_1\leqslant \nu_a \leqslant C_1$ with the same constant $C_1$. From this, as   in Step 1, we also get that $\|\nu_a\|_{H^1(\mathbb T^n)}\leqslant C_2$, with the same constant $C_2$. The Poincaré-Wirtinger inequality~\eqref{PWnTorus}, together with the lower and upper bounds on $\nu_a$ classically yields a Poincaré inequality for $\nu_a$. Indeed, for any $\varphi \in H^1(\nu_a)$, $\int_{\IT^n} \varphi \nu_a$ is the minimizer in $\IR$ of $c\mapsto \int_{\IT^n} (\varphi-c)^2 \nu_a$, so that 
\begin{align}
\int_{\IT^n} \po \varphi - \int_{\IT^n} \varphi \nu_a \pf^2 \nu_a  \ \leqslant \ \int_{\IT^n} \po \varphi - \int_{\IT^n} \varphi\pf^2 \nu_a  \
 &\leqslant \ C_1 \int_{\IT^n} \po \varphi - \int_{\IT^n} \varphi\pf^2\nonumber \\
 &\leqslant \ \frac{C_1}{4\pi^2}\int_{\IT^n} |\nabla \varphi |^2  \nonumber \\
 &\leqslant \ \frac{C_1^2}{4\pi^2}\int_{\IT^n} |\nabla \varphi |^2 \nu_a\,. \label{eq:Poincarenua}
\end{align}

\medskip

\noindent\textbf{Step 4:} We now proceed to the proof of the last part of the proposition, from which the uniqueness of $\nu_a$ immediately follows. Let $a,b\in L^p(\IT^n,\IR^n)$ with $L^p$ norms less than $M$ and $\nu_a,\nu_b\in H^1(\IT^n)$ be probability densities solutions of \eqref{eq:equilibreFPa} (with respective drift $a$ and $b$). From the $L^\infty$-bounds on $\nu_a$, $1/\nu_a$, $\nu_b$ and $1/\nu_b$ obtained in Step 3, we get that $\nu_b/\nu_a$ and  $(\nu_b/\nu_a)^2$ are in $H^1(\IT^n)$. Applying \eqref{eq:equilibreFPa} for $b$ with $\varphi = \nu_b/\nu_a$ as a test function,
\begin{align*}
0 \ &=\ \int_{\IT^n} b \cdot \nabla \po \frac{\nu_b}{\nu_a}\pf \nu_b  - \nabla \po \frac{\nu_b}{\nu_a}\pf \cdot \nabla \nu_b \\
 \ &= \ \int_{\IT^n} b \cdot \nabla \po \frac{\nu_b}{\nu_a}\pf \nu_b  - \nabla \po \frac{\nu_b}{\nu_a}\pf \cdot \nabla \left(\frac{\nu_b}{\nu_a} \, \nu_a \right) \\
&=\ \int_{\IT^n} b \cdot \nabla \po \frac{\nu_b}{\nu_a}\pf \nu_b  - \left|\nabla \po \frac{\nu_b}{\nu_a}\pf \right|^2 \nu_a - \frac12 \nabla \po \po \frac{\nu_b}{\nu_a}\pf^2 \pf \cdot \nabla \nu_a \\
&=\ \int_{\IT^n} b \cdot \nabla \po \frac{\nu_b}{\nu_a}\pf \nu_b  - \left|\nabla \po \frac{\nu_b}{\nu_a}\pf \right|^2 \nu_a - \frac12 a \cdot \nabla \po \po \frac{\nu_b}{\nu_a}\pf^2 \pf \nu_a\\
& = \ \int_{\IT^n} b \cdot \nabla \po \frac{\nu_b}{\nu_a}\pf \nu_b  - \left|\nabla \po \frac{\nu_b}{\nu_a}\pf \right|^2 \nu_a -  a \cdot \nabla \po  \frac{\nu_b}{\nu_a} \pf \nu_b
\end{align*} 
where the last term of the above equality stems from \eqref{eq:equilibreFPa} with drift $a$ and test function $\varphi = (\nu_b/\nu_a)^2/2$. As a consequence, using the Cauchy-Schwarz's inequality and the uniform bounds on $\nu_a$ and $\nu_b$, one gets:
\[ \left\|\nabla \po\frac{\nu_b}{\nu_a}\pf \right\|_{L^2(\nu_a)}^2 \  =  \ \int_{\IT^n} (b -a)\cdot \nabla \po \frac{\nu_b}{\nu_a}\pf \nu_b    \ 
 \leqslant \ C_1^2  \| b-a\|_{L^2(\IT^n)}\left\|\nabla \po\frac{\nu_b}{\nu_a}\pf \right\|_{L^2(\nu_a)}\,,\]
 i.e
\begin{equation}
 \left\| \nabla \po\frac{\nu_b}{\nu_a}\pf \right\|_{L^2(\nu_a)} \leq \ C_1^2  \| b-a\|_{L^2(\IT^n)}.\label{eq:bound-grad-nub-nua}
 \end{equation}
Now, since 
$$\left\| \nu_b - \nu_a \right\|_{L^{2}(\IT^n)}^2 \leq \|\nu_a\|_{\infty} \, \dst{\int_{\IT^n} \left| \frac{\nu_b}{\nu_a}-1\right|^{2} \, \nu_a} \leq C_1 \left\| \frac{\nu_b}{\nu_a}-1 \right\|_{L^{2}(\nu_a)}^2 ,$$
using the Poincaré inequality~\eqref{eq:Poincarenua} with $\varphi = \nu_b/\nu_a$ (so that $\int_{\IT^n} \varphi \nu_a = 1$) yields:
\begin{align*}
\left\| \nu_b - \nu_a \right\|_{L^{2}(\IT^n)}^2  & \leq  \frac{C_1^3}{4\pi^2} \left\| \nabla \left(\frac{\nu_b}{\nu_a}\right) \right\|_{L^{2}(\nu_a)}^2 \\
& \leq \frac{C_1^7}{4\pi^2}  \| b-a\|_{L^2(\IT^n)}^2 \quad \text{(using~\eqref{eq:bound-grad-nub-nua})}.
\end{align*}
Hence
$$ \|\nu_b-\nu_a\|_{L^2(\IT^n)} \ \leqslant \  \frac{C_{1}^{\frac{7}{2}}}{2\pi} \| b-a\|_{L^2(\IT^n)}\,.$$
In the particular case $a=b$, we get that there is only one probability density $\nu\in H^1(\IT^n)$ that solves \eqref{eq:equilibreFPa}.
\end{proof}

\subsection{Proof of Theorem \ref{ThmExistence}}
\label{subsec:proof-existence}

\begin{proof}
Let us recall that one can assume, without loss of generality, that $\beta = 1$ (see the change of variables at the begining of Section~\ref{sec:long-time-cv}). From now on, let us fix $p=n+1$. Consider \nv{$\mathcal{P}^{+}$ the set of probability densities on $\IT^n$  that are lower bounded by a positive constant.}
Given a probability measure $\pi \in \mathcal{P}^{+}$, let 
$$G_{\pi}(x)=\displaystyle{\int_{\mathbb{T}^{n-m}} -\mathcal{F}_{1}(x,y) \, \frac{\pi(x,y)}{\pi^{\xi}(x)} \, \dd y}, \quad \forall x \in \mathbb{T}^m,$$
 where $\pi^{\xi}(.)= \displaystyle{\int_{\mathbb{T}^{n-m}} \pi(.,y) \, \dd y}$. In the ABF case, set $B_\pi = G_\pi$ and, in the PABF case, consider the Helmholtz projection 
$$B_{\pi}=\nabla H_{\pi}=\mathsf{P}_{L^{2}(\lambda)} (G_{\pi}).$$
 In both cases, given \cite[Lemma 15.13]{Ambrosio}, for all $p\geq 2$, there exists a constant $c^* >0$ such that,
 \begin{equation}
 \|B_\pi\|_{L^{p}(\IT^m)} \leqslant c^{*} \|G_\pi\|_{L^{p}(\IT^m)} \leqslant c^{*}\|\mathcal F\|_\infty, \label{Dem-ThmExistence-BoundLp}
 \end{equation}
  in other words, for every $\pi \in \mathcal{P}^{+}$, $B_{\pi}$ belongs to the $L^p$ ball $E=\{ f\in L^p(\mathbb T^m),\ \|f\|_{L^{p}(\IT^m)} \leq c^* \|\mathcal{F}\|_{\infty} \}$. In return, given $B \in E$, consider the infinitesimal generator $\mathcal{L}_{B}= \left( \mathcal{F}+B\right) \cdot \nabla + \Delta$ and denote by $\pi_{B}  $ its invariant measure\nv{, such as given in Proposition~\ref{Prop:driftLp} (in particular $\pi_{B} \in \mathcal{P}^{+}$)}. Composing these two steps, we obtain an application from $E$ to itself,
$$\begin{array}{lclcl}
 
T&: & E & \longrightarrow & E \\
 & & f &\longmapsto & B_{\pi_{f}}
\end{array}.$$
The link with  Theorem~\ref{ThmExistence} is that a probability measure $\pi$ is a stationary state for the non-linear dynamics~\eqref{EqFokkerPlanckABF} if and only if the associated bias $B_{\pi}$ is a fixed point of $T$. Proving Theorem~\ref{ThmExistence} is thus equivalent to prove that $T$ admits a fixed point. This will be established thanks to the Schauder’s fixed point theorem~\cite[Part 9.2.2 Theorem 3]{EvansPDE}. One thus have to prove that $T$ is continuous on $(E,\|\cdot\|_{L^{p}(\IT^m)})$ and that the family $T(E):=\{T(B),\ B\in E\}$ has compact closure in $L^p$. We have already seen that $T(E) \subset E$, which is a bounded subset of $L^p$. From the Fréchet-Kolmogorov theorem \cite[Theorem IV.25]{Brezis}, compactness follows from the following condition :
\begin{equation}\label{eq:Frechet-Kol}
\sup_{z\in\IR^{n},|z|\leqslant \delta} \sup_{f\in T(E)}\|\tau_z f -f\|_{L^{p}(\IT^m)} \ \underset{\delta\rightarrow 0}\longrightarrow \ 0,
\end{equation}
where $\tau_z$ is the translation operator, namely $\tau_z f(x) = f(x+z)$ for all $x\in\mathbb T^m$.

Let us recall that from \nv{Proposition~\ref{Prop:driftLp} } there exists a constant $C >0$ such that for all $B\in E$,
\begin{equation}
\|\pi_B\|_{H^1(\mathbb T^n)}+\|\pi_B\|_{\infty}+\|1/\pi_B\|_\infty \leqslant C \label{eq:Dem-Thm-Existence-bound-inv-mes}
\end{equation}
and for all $B_1,B_2\in E$,
\begin{equation}
\|\pi_{B_1} - \pi_{B_2}\|_{L^2(\mathbb T^n)} \ \leq \ C \|B_1-B_2\|_{L^2(\IT^m)}\,.\label{eq:Dem-Thm-Existence-bound-diff-mes}
\end{equation}

\bigskip 

\noindent{\textbf{Continuity of $T$}}.  Let $B_1,B_2\in E$ and to alleviate notations, denote by $\pi_1 = \pi_{B_1}$, $\pi_2=\pi_{B_2}$ the associated invariant measures. In both the ABF and PABF cases, using the same arguments as in~\eqref{Dem-ThmExistence-BoundLp} one gets:
 $$\|T(B_1)-T(B_2)\|_{L^{p}(\IT^m)} \leqslant c^* \|G_{\pi_{1}}-G_{\pi_{2}}\|_{L^{p}(\IT^m)}.$$
Moreover, relying on inequalities~\eqref{eq:Dem-Thm-Existence-bound-inv-mes} and ~\eqref{eq:Dem-Thm-Existence-bound-diff-mes}, one has, for all $x\in\mathbb T^m$,
\begin{align*}
 |G_{\pi_1}(x)-G_{\pi_2}(x)| 
& \leq  \|\mathcal F\|_{\infty}   \displaystyle{\int_{\mathbb{T}^{n-m}}\left| \frac{ \pi_{1}(x,y)}{\pi_1^\xi(x)} - \frac{\pi_{2}(x,y)}{\pi_2^\xi(x)} \right | \, \dd y} \\
  & \leq  \|\mathcal F\|_{\infty}\int_{\mathbb{T}^{n-m}}  \frac{|\pi_1(x,y)-\pi_2(x,y)|}{\pi^\xi_1(x)}+\frac{\pi_2(x,y)|\pi_1^\xi(x)-\pi_2^\xi(x)|}{\pi_1^\xi(x) \pi_2^\xi(x)} \, \dd y\\
& \leq  \|\mathcal F\|_{\infty}C^3\int_{\mathbb{T}^{n-m}}   |\pi_1(x,y)-\pi_2( x,y)|+|\pi_1^\xi(x)-\pi_2^\xi(x)|  \dd y\\
& \leq  2\|\mathcal F\|_{\infty}  C^3\int_{\mathbb{T}^{n-m}}   |\pi_1(x,y)-\pi_2(x,y)|  \dd y\\
&\leq 2\|\mathcal F\|_{\infty}  C^3 \|\pi_{1}-\pi_{2} \|_{L^{2}(\IT^n)}\\
& \leq  2\|\mathcal F\|_{\infty}  C^4 \|B_{1}-B_{2} \|_{L^{2}(\IT^m)}\,.
\end{align*}
As a consequence, since $p\ge 2$, by Sobolev embedding,
$$\|T(B_1)-T(B_2)\|_{L^{p}(\IT^m)} \leq  c^* \|G_{\pi_{1}}-G_{\pi_{2}}\|_{L^{p}(\IT^m)} \leq  \ 2 c^*\|\mathcal F\|_{\infty} C^4 \|B_1-B_2\|_{L^{p}(\IT^m)}\,,$$
which proves that $T$ is a Lipschitz function on $(E,\|\cdot\|_{L^{p}(\IT^m)})$.\\

\begin{rem}
\nv{In the particular case where $\|\mathcal{F} \|_{\infty}$ is small enough so that  $2 c^*\|\mathcal F\|_{\infty} C^4<1$, we directly get that $T$ is a contraction of the $L^p$-norm, which yields the existence and uniqueness of a fixed-point.}
\end{rem}

\bigskip

\noindent{\textbf{Compactness}}. Fix $B\in E$ and let $\pi=\pi_B$ to alleviate notations. For $z\in \mathbb R^m$, $\tau_z$ commutes with the Helmholtz projection so that, using \cite[Lemma 15.13]{Ambrosio},
\[\|\tau_z\mathsf{P}_{L^{2}(\lambda)} (G_{\pi}) - \mathsf{P}_{L^{2}(\lambda)} (G_{\pi})\|_{L^{p}(\IT^m)} \ = \ \|\mathsf{P}_{L^{2}(\lambda)} (\tau_z G_{\pi}-  G_{\pi})\|_{L^{p}(\IT^m)} \ \leq \ c^* \|\tau_z G_{\pi}-  G_{\pi}\|_{L^{p}(\IT^m)}\,.\]
Hence, in both the ABF and PABF cases, for all $z\in\mathbb R^m$,
\[\|\tau_z T(B) - T(B)\|_{L^{p}(\IT^m)} \ \leq \ c^* \|\tau_z G_\pi - G_\pi\|_{L^{p}(\IT^m)}\,.\]
Now, for all $x\in \mathbb T^m$ and $z\in \mathbb R^m$, using the same argument as in the proof of the continuity of $T$,
\begin{align*}
 |G_{\pi}(x+z)-G_{\pi}(x)| &= \left| \dst{\int_{\IT^{n-m}} -\mathcal{F}_{1}(x+z,y) \frac{\pi(x+z,y)}{\pi^{\xi}(x+z)} \, \dd y} -\dst{\int_{\IT^{n-m}} -\mathcal{F}_{1}(x,y) \frac{\pi(x,y)}{\pi^{\xi}(x)} \, \dd y} \right|\\
 &\leq \left|\dst{\int_{\IT^{n-m}}\left(-\mathcal{F}_{1}(x+z,y)+ \mathcal{F}_{1}(x,y) \right) \frac{\pi(x+z,y)}{\pi^{\xi}(x+z)} \, \dd y}\right| \\
 &\quad +\left|\dst{\int_{\IT^{n-m}} -\mathcal{F}_{1}(x,y) \left(\frac{\pi(x+z,y)}{\pi^{\xi}(x+z)}-\frac{\pi(x,y)}{\pi^{\xi}(x)}\right) \, \dd y} \right|\\
& \leq   |z|\| \nabla \mathcal F\|_{\infty} +  \|\mathcal F\|_\infty \displaystyle{\int_{\mathbb{T}^{n-m}}\left| \frac{ \pi(x+z,y)}{\pi^\xi(x+z)} - \frac{\pi(x,y)}{\pi^\xi(x)} \right | \, \dd y} \\
  & \leq  |z|\| \nabla \mathcal F\|_{\infty} + 2  \|\mathcal F\|_\infty   C^3\int_{\mathbb{T}^{n-m}}   |\pi(x+z,y)-\pi(x,y)|  \dd y\,\\
  &\leq  |z|\| \nabla \mathcal F\|_{\infty} + 2  \|\mathcal F\|_\infty   C^3 \|\tau_z \pi -\pi\|_{L^{2}(\IT^n)},
\end{align*}
where $C$ stems from ~\eqref{eq:Dem-Thm-Existence-bound-inv-mes} and ~\eqref{eq:Dem-Thm-Existence-bound-diff-mes}.
To bound the last term, write
\begin{align*}
\int_{\mathbb{T}^{n}}   |\pi(x+z,y)-\pi(x,y)|^2 \dd x  \dd y & =  \int_{\mathbb{T}^{n}}   \left|\int_{0}^1 z \cdot \nabla_{x} \pi(x+sz,y) \dd s\right|^2  \dd x \dd y\\
& \leq    \int_{0}^1 \int_{\mathbb{T}^{n}}   | z|^2  | \nabla_{x} \pi(x+sz,y) |^2  \dd x \dd y \dd s \\
& =  |z|^2 \|\nabla_{x} \pi\|_2^2\\
&\leq |z|^{2} \|\nabla \pi\|_{2}^{2}\,.
\end{align*}
As a conclusion, using ~\eqref{eq:Dem-Thm-Existence-bound-inv-mes}:
\[\|\tau_z G_\pi - G_\pi\|_{L^{p}(\IT^n)}  \ \leq  \  |z| \left( \| \nabla \mathcal F\|_{\infty}     + 2\|\mathcal F\|_\infty    C^4\right)\,, \]
so that \eqref{eq:Frechet-Kol} holds.

Consequently, there exists an equilibrium measure $\pi_{\infty}^{\mathcal{F}}$ which is continuous and positive, along with an associated bias $B_{\infty}^{\mathcal{F}}$. By   \nv{Proposition~\ref{Prop:driftLp}, one has positive upper and lower bounds on $\pi_{\infty}^{\mathcal{F}}$ and, relying on the Holley-Stroock perturbation result} \cite[Proposition 5.1.6]{BakryGentilLedoux}, $\pi_{\infty}^{\mathcal{F}}$ satisfies $LSI(R)$ for some $R>0$ and the conditional densities $y\mapsto \pi_{\infty,x}^{\mathcal{F}}(y):=\pi_{\infty}^{\mathcal{F}}(x,y)/\pi_{\infty}^{\mathcal{F},\xi}(x)$ satisfy $LSI(\rho)$ with some $\rho>0$ uniform with respect to $x \in \IT^m$.
\end{proof}

\subsection{Proof of Proposition~\ref{PropPerturb}}
\label{subsec:perturb-diffusion}
Let us conclude Section~\ref{sec:existence-equilibrium} with the proof of Proposition~\ref{PropPerturb}.
\begin{proof}
Let us consider the PABF algorithm. Again, without loss of generality, we suppose that $\beta = 1$. Fix $V\in\mathcal C^2(\mathbb T^n)$, and define $$\mathfrak{F}=\{(\mathcal F,\pi_\infty^\mathcal F)\in\mathcal C^1(\mathbb T^n, \IR^n)\times\mathcal P(\mathbb T^n)\, |\, \|\mathcal F+\nabla V\|_\infty \leq 1,\ \pi_\infty^\mathcal F\text{ stationary state for~\eqref{EqFokkerPlanckABF}}\}.$$ In particular, for $(\mathcal F,\pi_\infty^\mathcal F)\in\mathfrak F$, $\pi_\infty^\mathcal F$ is the invariant measure of the diffusion \eqref{EqA} on $\mathbb T^n$ with drift $a = \mathcal F + \nabla (H_\mathcal F\circ \xi)$. Moreover,
\begin{align*}
\|\mathcal F +\nabla V\|_\infty \leq 1& \Rightarrow  \|\mathcal F\|_\infty \leq 1+\|\nabla V\|_\infty \\
& \Rightarrow  \|G_{\mathcal F}\|_{\infty} \leq 1+\|\nabla V\|_\infty  .
\end{align*}
By \cite[Lemma 15.13]{Ambrosio}, for all $p\ge 2$, there exists $c^{*}>0$ such that 
$$\|\nabla H_{\mathcal{F}} \|_{L^{p}(\IT^{m})} \leq c^{*} \|G_{\mathcal F}\|_{L^{p}(\IT^{m})} \le c^{*} \left( 1+\|\nabla V\|_\infty \right),$$
which yields, by Minkowski's inequality, for all $p\ge 2$
\begin{equation}
\|\mathcal F + \nabla (H_\mathcal F\circ \xi) \|_{L^{p}(\IT^{n})} \leq \left(c^{*}+1\right) \left(1+\|\nabla V\|_\infty \right).\label{Dem-PropPerturb-Bound1}
\end{equation}
Note on the other hand, that for all $p\ge 2$
\begin{equation}
\|-\nabla V + \nabla (H_\mathcal F\circ \xi) \|_{L^{p}(\IT^{n})} \leq  (1+c^*)\|\nabla V\|_{\infty} +c^*. \label{Dem-PropPerturb-Bound2}
\end{equation}
Denote by $\nu_\mathcal F$ the invariant measure of the diffusion \eqref{EqA} on $\mathbb T^n$ with drift $a = -\nabla V + \nabla (H_\mathcal F\circ \xi)$, in other words
$$\nu_\mathcal F(x,y) \ = \ \frac{1}{Z_{\nu_\mathcal F}} e^{-V(x,y)+H_\mathcal F(x)}\,,\qquad Z_{\nu_\mathcal F} = \int_{\mathbb T^n} e^{-V(u,v)+H_\mathcal F(u)}\dd u \dd v\,.$$
In the rest of the proof $(\mathcal F,\pi_\infty^\mathcal F)\in\mathfrak{F}$ is fixed and we are careful to give bounds which are uniform over $\mathfrak{F}$.
Besides, to alleviate notations, we simply denote by $\pi = \pi_\infty^\mathcal F$, $\nu=\nu_\mathcal F$, $H=H_\mathcal F$ and $G=G_\mathcal F$.

 Given the bounds \eqref{Dem-PropPerturb-Bound1} and \eqref{Dem-PropPerturb-Bound2}, one can apply \nv{Proposition~\ref{Prop:driftLp}} with a drift $a$ equal to either $\mathcal F + \nabla (H_\mathcal F\circ \xi)$ or $-\nabla V +\nabla (H_\mathcal F\circ \xi)$, which are both bounded in $L^{p}(\IT^n)$ for all $p\geq 1$ as shown above. As a consequence, there exists a constant $C>0$ such that for all $(\mathcal{F},\pi) \in \mathfrak{F}$,
$$\|\nu\|_\infty + \|1/\nu\|_\infty + \|\nu\|_{ H^1(\mathbb T^n)} + \|\pi \|_\infty + \|1/\pi\|_\infty + \|\pi\|_{ H^1(\mathbb T^n)} \ \leq \ C,$$
and
\begin{equation}\label{EqProp3nupi}
\|\pi - \nu\|_{L^{2}(\mathbb T^n)}\leqslant C\|\mathcal F + \nabla V\|_{L^{2}(\IT^{n})} \leqslant  C \|\mathcal F + \nabla V\|_\infty\,.
\end{equation}

Notice that $\nu$ has the same conditional laws (given $x$) than the Gibbs measure $\mu$, so that
$$ \nabla A(x) \ = \ \frac{\int_{\mathbb T^{n-m}} \nabla_x V(x,y) e^{-V(x,y)} \dd y}{\int_{\mathbb T^{n-m}}   e^{-V(x,y)} \dd y} \ = \ \int_{\mathbb T^{n-m}} \nabla_x V(x,y) \frac{\nu  (x,y)}{\nu^\xi(x)} \dd y\,.$$
As a consequence,
\begin{align*}
|\nabla A(x) - G(x)| & =  \left| \int_{\mathbb T^{n-m}} \nabla_x V(x,y) \frac{\nu (x,y)}{\nu ^\xi(x)} \dd y + \int_{\mathbb T^{n-m}} \mathcal F_1(x,y) \frac{\pi(x,y)}{\pi^\xi(x)} \dd y \right| \\
& \leq  \|\mathcal F + \nabla V\|_\infty+ \|\nabla V\|_\infty \int_{\mathbb T^{n-m}} \left| \frac{\nu (x,y)}{\nu ^\xi(x)} - \frac{\pi(x,y)}{\pi^\xi(x)}\right| \dd y .
\end{align*}
Using the same argument as in the proof of the continuity of $T$ in Theorem~\ref{ThmExistence} and~\eqref{EqProp3nupi},
\begin{align*}
\int_{\mathbb T^{n-m}} \left| \frac{\nu (x,y)}{\nu ^\xi(x)} - \frac{\pi(x,y)}{\pi^\xi(x)}\right| \dd y & \leq 2C^3 \int_{\mathbb T^{n-m}} \left|  \nu (x,y) -  \pi(x,y) \right| \dd y\\
& \leq  2C^3 \|\nu-\pi\|_{ L^2(\mathbb T^n)}\\
& \leq  2C^4 \|\mathcal F + \nabla V\|_{\infty}\,.
\end{align*}
We have thus obtained that, uniformly over $\mathfrak{F}$,
$$\|\nabla A-G\|_{L^{p}(\IT^{m})} \ \leq \ (1+2\|\nabla V\|_\infty C^4)\|\mathcal F + \nabla V\|_{\infty}\,.$$
Which yields, given \cite[Lemma 15.13]{Ambrosio}:
\begin{equation}
\|\nabla A-\nabla H\|_{L^{p}(\IT^{m})} = \|\mathsf{P}_{L^{2}(\lambda)}\left( \nabla A - G \right) \|_{L^{p}(\IT^{m})} \ \leq c^{*} \|\nabla A-G\|_{L^{p}(\IT^{m})} \le K_{V} \|\mathcal F + \nabla V\|_{\infty}.  \label{eq:PropPerturb-bound-bias}
\end{equation}
with $K_V=c^{*}(1+2\|\nabla V\|_\infty C^4)$.  This concludes the proof of the first point of Proposition~\ref{PropPerturb}. Concerning the second point, first note that
\[\hat I_{\psi} \ =\ \frac{\int_{\mathbb T^n} \psi(x,y) e^{- H (x)} \pi (x,y) \dd x \dd y }{\int_{\mathbb T^n}e^{-  H (x)} \pi (x,y)  \dd x \dd y} \  = \ \frac{\int_{\mathbb T^n} \psi(x,y) e^{-  H (x)} \pi (x,y) \dd x \dd y }{\int_{\mathbb T^m}e^{-  H (x)}   \dd x },\]
where we used Proposition~\ref{PropHistoPlat} to see that since $\pi$ is a stationary state of~ \eqref{EqFokkerPlanckABF},  $\pi^\xi$ is necessarily the uniform measure on $\mathbb T^m$. Notice that this expression is unchanged if $H$ is replaced by $H+c$ for some constant $c>0$. As a consequence, for the remainder of the proof and without loss of generality, we suppose that $H$ is normalised so that $\int_{\IT^m} e^{-H} = 1$.

Using that
\[\int_{\mathbb T^n} \psi \dd \mu \ = \ \frac{Z_\nu}{Z_\mu}\int_{\mathbb T^n} \psi e^{-H\circ\xi }\dd \nu\,,\]
we are led to
\begin{align}
\left| \int_{\mathbb T^n} \psi \dd \mu -\hat I_{\psi} \right|  & = \left| \frac{Z_\nu}{Z_\mu}  \int_{\mathbb T^n} \psi e^{- H\circ\xi } \dd \nu  - \int_{\mathbb T^n} \psi e^{- H\circ\xi } \dd \pi \right|\notag\\
& \leq  \left\|\psi e^{-  H\circ\xi } \right\|_\infty \left( \left| \frac{Z_\nu}{Z_\mu}     -  1 \right| + \|\nu-\pi\|_{L^1(\mathbb T^n)}\right) \notag\\
&\leq \left\|\psi e^{-  H\circ\xi } \right\|_\infty \left( \left| \frac{Z_\nu}{Z_\mu}     -  1 \right| + \|\nu-\pi\|_{ L^2(\mathbb T^n)}\right) \,.
\label{EqProp3Ipsi} 
\end{align}
Besides,
\[\frac{Z_\nu}{Z_\mu} \ = \ \frac{\int_{\mathbb T^n} e^{-(V(x,y)-H (x) )}\dd x \dd y }{\int_{\mathbb T^n} e^{-V(x,y)}\dd x \dd y} \ = \ \frac{\int_{\IT^m} e^{-(A(x)-H (x))}\dd x }{\int_{\IT^m} e^{- A(x)}\dd x} \,.\]

Again, this expression is unchanged if $A$ is replaced by $A+c$ for some constant $c>0$. In the remaining of the proof and without loss of generality, we suppose that $A$ is normalized so that $\int_{\IT^m}A-H=0$. As a consequence, by the Poincaré-Wirtinger inequality~\cite[Part 5.8.1 Theorem 1]{EvansPDE}, there exists a constant $\bar{K}>0$ \nv{(that depends only on $m$ and $p$)} such that:
$$\|A-H\|_{L^p(\IT^m)} \le \bar{K}\|\nabla A -\nabla H \|_{L^{p}(\IT^m)}.$$
Thus, using~\eqref{eq:PropPerturb-bound-bias}:
\begin{align*}
\|A-H\|_{W^{1,p}(\IT^m)}&= \|A-H\|_{L^p(\IT^m)} +\|\nabla A - \nabla H\|_{L^p(\IT^m)}\\
&\leq  (\bar{K}+1)K_{V} \|\mathcal F + \nabla V\|_{\infty}.
\end{align*}
Now, (\ref{Morrey}) yields the existence of $\mathcal{K}>0$ such that
$\|A-H\|_{\infty} \le \mathcal{K} \|A-H\|_{W^{1,p}(\IT^{m})}, $
hence
$$\|A-H\|_{\infty} \le \tilde{K}_{V} \|\mathcal F + \nabla V\|_{\infty},$$
where $\tilde{K}_{V}:= \mathcal{K}(\bar{K}+1)K_{V} =\mathcal{K}(\bar{K}+1)c^{*}(1+2\|\nabla V\|_\infty C^4)$. 
Then, using that $|e^a-1| \leqslant |a|e^{|a|}$ for all $a\in \mathbb R$, for all $x\in \mathbb T^m$,
\[ |e^{-A(x)+H(x)}-1| \ \leqslant \   \tilde{K}_V\|\mathcal F + \nabla V\|_\infty e^{\tilde{K}_{V}\|\mathcal F + \nabla V\|_{\infty}}\,,\]
so that, using the fact that $\dst{\int_{\IT^m} e^{-H}} =1$,
\begin{align*}
\left| \frac{Z_\nu}{Z_\mu}     -  1 \right|=\left|\frac{\int_{\IT^m}e^{-A+H}}{\int_{\IT^m}e^{-A}}-1\right| & \leqslant  \frac{\left|\int_{\IT^m}e^{-A+H}- 1 \right| + \left|\int_{\IT^m}e^{-A}- 1 \right| }{\int_{\IT^m}e^{-A}}\\
& \leqslant  \frac{\int_{\IT^m}\left| e^{-A+H}- 1 \right| + \int_{\IT^m} e^{-H}|e^{-A+H}-1|}{\int_{\IT^m}e^{-H-|H-A|}} \\
& \leqslant   2\tilde{K}_V \|\mathcal F + \nabla V\|_\infty e^{2 \tilde{K}_V\|\mathcal F + \nabla V\|_\infty}\\
&\leq  2 \tilde{K}_V \|\mathcal F + \nabla V\|_\infty e^{2 \tilde{K}_V}.
\end{align*}
Combining this with \eqref{EqProp3nupi} in \eqref{EqProp3Ipsi}, we have obtained that
\begin{align*}
\left| \int_{\mathbb T^n} \psi \dd \mu -\hat I_{\psi} \right| & \leqslant  \frac{\left\|\psi e^{-  H\circ\xi } \right\|_\infty}{\int_{\mathbb T^m} e^{-H}} \left( 2\tilde{K}_V e^{2\tilde{K}_V} + C\right)\|\mathcal F + \nabla V\|_\infty\\
&\leqslant  \frac{\left\|\psi e^{-  A\circ\xi } \right\|_\infty }{\int_{\mathbb T^m} e^{-A} } e^{2\tilde{K}_V\|\mathcal F + \nabla V\|_\infty}\left( 2\tilde{K}_V e^{2\tilde{K}_V} + C\right)\|\mathcal F + \nabla V\|_\infty\\
& \leqslant  \frac{\left\|\psi  \right\|_\infty \left\| e^{-  A  } \right\|_\infty }{\int_{\mathbb T^m} e^{-A} } e^{2\tilde{K}_V}\left( 2K_V e^{2\tilde{K}_V} + C\right)\|\mathcal F + \nabla V\|_\infty\,,
\end{align*}
which yields the conclusion.

\end{proof}

\section{Long-time convergence}\label{sec:long-time-cv}
In Section~\ref{subsec:long-time-intermediate-results} one can find the proof of intermediate results that will prove useful for the proofs of Theorem~\ref{ThmConservative}, Theorem~\ref{ThmCVtempslong}, and Corollary~\ref{CorCvBias}. Said proofs can respectively be found in Section~\ref{subsec:thm-conservative}, Section~\ref{subsec:thm-non-conservative} and Section~\ref{subsec:corollary-cv}. \\

\noindent In all this section, \nv{to alleviate notations, we will denote by $\pi_\infty$ (dropping the $\mathcal F$ superscript)} a stationary measure given by Theorem~\ref{ThmExistence}. First and foremost, let us introduce the concept of total entropy and its macroscopic-microscopic decomposition. We define the \textit{total entropy} as:
$$E(t)=\mathcal{H}(\pi_{t}|\pi_{\infty}).$$
In the same manner, the entropy between the marginals in $x \in \IT^{m}$ (called \textit{macroscopic entropy} henceforth) is given by:
$$E_{M}(t)=\mathcal{H}(\pi^{\xi}_{t}|\pi^{\xi}_{\infty}).$$
Note that accordingly, one can define the \textit{macroscopic Fisher information}:
$$I_{M}(t)=I(\pi^{\xi}_{t}|\pi^{\xi}_{\infty}).$$
The entropy between the conditional measures at a given $x \in \mathbb{T}^{m}$ (called \textit{local entropy} in the following) is:
$$e_{m}(t,x)= \mathcal{H}(\pi_{t,x} | \pi_{\infty,x}),$$
where $\pi_{t,x}(.)=\displaystyle{\frac{\pi_{t}(x,.)}{\pi_{t}^{\xi}(x)}}$ and  $\pi_{\infty,x}(.)=\displaystyle{\frac{\pi_{\infty}(x,.)}{\pi_{\infty}^{\xi}(x)}}$. Now, let us introduce the so-called \textit{microscopic entropy}:
$$E_{m}(t)= \displaystyle{\int_{\IT^m} \, e_{m}(t,x) \pi_t^{\xi}(x) \dd x}. $$
One has for all $t \ge 0$, $E(t)=E_{m}(t)+E_{M}(t)$  (see \cite[Lemma 1]{LRS07}).\\

Note that we have the following bound on the microscopic entropy:
\begin{align*}
E_{m}(t)=\displaystyle{\int_{\mathbb{T}^{m}} e_{m}(t,x)\, \pi_{t}^{\xi}(x) \, \dd x} &= \displaystyle{\int_{\mathbb{T}^{m}} \mathcal{H}(\pi_{t,x}|\pi_{\infty,x})\, \pi_{t}^{\xi}(x) \, \dd x}\\
&\leq  \frac{1}{2\rho}\displaystyle{\int_{\mathbb{T}^{m}} I(\pi_{t,x}|\pi_{\infty,x})\, \pi_{t}^{\xi}(x) \, \dd x} \quad \text{(using Theorem~\ref{ThmExistence} (ii))}.
\end{align*}
Since $\dst{\nabla_{y} \ln \po\frac{\pi_{t,x}}{\pi_{\infty,x}}\pf=\nabla_{y}\ln \po\frac{\pi_t}{\pi_{\infty}}\pf}$, this leads to

\begin{equation}
E_{m}(t) \leq \frac{1}{2\rho}\displaystyle{\int_{\mathbb{T}^{n}} \left|\nabla_{y}\ln \po\frac{\pi_t}{\pi_{\infty}}\pf\right|^{2} \,\pi_t  }. \label{BoundEm}
\end{equation}

\subsection{Intermediate results}
\label{subsec:long-time-intermediate-results}

The proofs of both Theorems~\ref{ThmConservative} and \ref{ThmCVtempslong} will rely on the following intermediate results. Assumptions~\ref{AssuCINI} and~\ref{HypoGeneral} are enforced. Here, $\mathcal{F}$ can be either conservative ($\mathcal{F}=-\nabla V$) or not\nv{, and $\left(\pi_{\infty},B_\infty ,G_\infty\right)$   denotes a stationary state of~\eqref{EqFokkerPlanckABF}, with $R,\rho$ the corresponding constants given by Theorem~\ref{ThmExistence}}.

\begin{lem}[Bound on $G_{t}(x)-G_\infty(x)$]\label{lem:BoundBias}
For all $t \ge 0$ and $x \in \IT^m$:
\begin{equation*}
 |G_{t}(x)-G_\infty(x)| \leq M \sqrt{\frac{2}{\rho} e_{m}(t,x)}. 
\end{equation*}
\end{lem}
\begin{proof}Note that, given Theorem~\ref{ThmExistence} (ii), since $\pi_{\infty,x}$ satisfies a log-Sobolev inequality with constant~$\rho>0$, it also satisfies a Talagrand inequality with constant $\rho$. Now, let $ x \in \mathbb{T}^m$ and $ \nu_{x} \in \Pi(\pi_{t,x},\pi_{\infty,x})$ be a coupling measure. Then, one has:
\begin{align*}
G_{t}(x)-  G_\infty(x) &= \displaystyle{\int_{\mathbb{T}^{n-m}} (-\mathcal{F}_{1}(x,y)+\mathcal{F}_{1}(x,y')) \, \nu_{x}(\dd y, \dd y') }\\
&\leq  M  \displaystyle{\int_{\mathbb{T}^{n-m}}|y-y'|\, \nu_{x}(\dd y, \dd y') } \quad \text{(by Assumption \ref{HypoGeneral})}\\
&\leq  M  \left(\displaystyle{\int_{\mathbb{T}^{n-m}}|y-y'|^{2}\, \nu_{x}(\dd y, \dd y') }\right)^{\frac{1}{2}}.
\end{align*}
Taking the infinimum over $\Pi(\pi_{t,x},\pi_{\infty,x})$ yields:
\begin{align*}
G_{t}(x)-  G_\infty(x) &\leq  M  W_2(\pi_{t,x},\pi_{\infty,x})\quad \\
&\leq  M    \sqrt{\frac{2}{\rho}\mathcal{H}(\pi_{t,x}|\pi_{\infty,x})} \quad (\text{by the Talagrand inequality~\eqref{eq:Talagrand-ineq}}).
\end{align*}
This yields the conclusion, since $\mathcal{H}(\pi_{t,x}|\pi_{\infty,x})=e_{m}(t,x)$.
\end{proof}

\begin{lem}[Total entropy]\label{lem:dEtot}
One has,
\begin{equation*}
\frac{\dd E}{\dd t } =- \displaystyle{\int_{\mathbb{T}^{n}} |\nabla \ln \po\frac{\pi_{t}}{ \pi_{\infty}}\pf|^{2} \, \pi_{t}}+\displaystyle{\int_{\mathbb{T}^{n}} \left(B_{t}- B_{\infty} \right)(x) \cdot \nabla_{x} \ln \po\frac{\pi_{t}}{ \pi_{\infty}}\pf \,\pi_{t}} .
\end{equation*}
\end{lem}
\begin{proof}
If $\mathcal{L}_{t}$ denotes the infinitesimal generator of \eqref{F-ABF} and $ \mathcal{L}_{t}'$ its formal adjoint in $L^{2}(\IT^n)$ then the Fokker-Planck equation \eqref{EqFokkerPlanckABF} can be rewritten as follows:
$$\partial_{t} \pi_{t} = \mathcal{L}_{t}'(\pi_{t}).$$

Denote by $\mathcal{L}_{\infty} \nv{=\mathcal F\cdot \nabla + B_\infty\cdot \nabla_x + \Delta}$ the infinitesimal generator associated to the stationary state $\left(\pi_{\infty}, B_{\infty} \right)$. One has:
\begin{align*}
\frac{\dd E}{\dd t } 
&= \displaystyle{\int_{\mathbb{T}^{n}} \partial_{t} \pi_{t} } + \displaystyle{\int_{\mathbb{T}^{n}}  \partial_{t} \pi _{t} \ln \po\frac{\pi_{t}}{\pi_{\infty}}\pf  }\\ 
&=\displaystyle{\int_{\mathbb{T}^{n}} \mathcal{L}_{t}'(\pi_{t}) \ln \po\frac{\pi_{t}}{ \pi_{\infty}}\pf  } \qquad \text{(since $\displaystyle{\int_{\mathbb{T}^{n}} \partial_{t} \pi_{t} }=0$)}\\
&=\displaystyle{\int_{\mathbb{T}^{n}} \mathcal{L}_{t}\left(\ln \po \frac{\pi_{t}}{ \pi_{\infty}}\pf\right)\pi_{t}   }\\
&=\displaystyle{\int_{\mathbb{T}^{n}} \left(\mathcal{L}_{\infty}+\mathcal{L}_{t}-\mathcal{L}_{\infty}\right)\left(\ln\po\frac{\pi_{t}}{ \pi_{\infty}}\pf\right)\pi_{t}   }\\
&= \displaystyle{\int_{\mathbb{T}^{n}}\mathcal{L}_{\infty}\left(\ln\po \frac{\pi_{t}}{ \pi_{\infty}}\pf\right) \pi_{t}}+ \displaystyle{\int_{\mathbb{T}^{n}}\left(\mathcal{L}_{t}-\mathcal{L}_{\infty}\right)\left(\ln \po\frac{\pi_{t}}{ \pi_{\infty}}\pf\right)\pi_{t}  }.
\end{align*}
Since $\mathcal{L}_{\infty}$ is the infinitesimal generator of a diffusion, it follows that, for any given functions $a$ and $f$:
$$\mathcal{L}_{\infty}(a(f))=a'(f)\mathcal{L}_{\infty}(f)+a"(f) |\nabla f|^{2},$$
as mentioned in \cite[Part 2.3]{MonmarcheGammaCalcul}. Applying this with $a(.)=\ln (.)$ and $f=\displaystyle{\frac{\pi_{t}}{ \pi_{\infty}}}$ we respectively obtain:
\begin{align*}
\displaystyle{\int_{\mathbb{T}^{n}}\mathcal{L}_{\infty}\po \ln \po \frac{\pi_{t}}{ \pi_{\infty}}\pf \pf \pi_{t}}&=\displaystyle{\int_{\mathbb{T}^{n}} \left( \frac{\pi_{\infty}}{ \pi_{t}}\cdot \mathcal{L}_{\infty}\po\frac{\pi_{t}}{ \pi_{\infty}}\pf - \left(\frac{\pi_{\infty}}{ \pi_{t}} \right)^{2}\cdot \left|\nabla \frac{\pi_{t}}{ \pi_{\infty}} \right|^{2} \right) \, \pi_t}\\
&= \displaystyle{\int_{\mathbb{T}^{n}}\mathcal{L}_{\infty}\left(\frac{\pi_{t}}{ \pi_{\infty}}\right)\, \pi_{\infty}} - \displaystyle{\int_{\mathbb{T}^{n}} \left|\nabla \ln \po\frac{\pi_{t}}{ \pi_{\infty}}\pf\right|^{2} \, \pi_{t}}\\
&=- \displaystyle{\int_{\mathbb{T}^{n}} \left|\nabla \ln \po \frac{\pi_{t}}{ \pi_{\infty}}\pf\right|^{2} \, \pi_{t}} \qquad \text{(since $\pi_{\infty}$ is invariant for $\mathcal{L}_{\infty}$)}
\end{align*}
and $$\displaystyle{\int_{\mathbb{T}^{n}}\left(\mathcal{L}_{t}-\mathcal{L}_{\infty}\right)\left(\ln \po\frac{\pi_{t}}{ \pi_{\infty}}\pf\right)\pi_{t}  }=\displaystyle{\int_{\mathbb{T}^{n}} (B_{t}-B_{\infty}) \cdot \nabla_{\nv{x}} \ln \po\frac{\pi_{t}}{ \pi_{\infty}}\pf \,\pi_{t}},$$
which concludes the proof.
\end{proof}

\subsection{Proof of Theorem \ref{ThmConservative}}
\label{subsec:thm-conservative} 

Let us prove the convergence of the ABF and PABF algorithms in the conservative case, namely when $\mathcal{F}=-\nabla V$. In that case $\pi_\infty = \mu_A$ is invariant by \eqref{EqFokkerPlanckABF} (recall $\mu_A$ is given by~\eqref{eq:mu-FE}), with a corresponding $G_\infty = \nabla A$, so that $B_\infty = \nabla A$ in both the ABF and PABF case.


\begin{lem}\label{lem:ExpressionGt-B}
In the conservative case ($\mathcal{F}=-\nabla V$ and $\pi_\infty = \mu_A$), for all $t \ge 0$ and $x \in \IT^m$:
 \begin{equation*}
G_{t}(x)-\nabla A(x)=\displaystyle{\int_{\mathbb{T}^{n-m}} \nabla_{x} \ln \po\frac{\pi_{t}(x,y)}{\pi_{\infty}(x,y)}\pf \frac{\pi_{t}(x,y)}{\pi^{\xi}_{t}(x)} \, \dd y} - \nabla_{x} \ln \po\frac{\pi^{\xi}_{t}(x)}{\pi_{\infty}^{\xi}(x)}\pf. 
\end{equation*}
\end{lem}
\begin{proof} Knowing that $\pi_{\infty}^{\xi}=1$, one has, for a fixed $x$ in $\IT^m$:
\begin{align*}
\lefteqn{\displaystyle{\int_{\mathbb{T}^{n-m}} \nabla_{x} \ln \po\frac{\pi_{t}}{\pi_{\infty}}\pf \frac{\pi_{t}}{\pi^{\xi}_{t}} \, \dd y} - \nabla_{x} \ln \po\frac{\pi^{\xi}_{t}}{\pi_{\infty}^{\xi}}\pf}\\
&= \displaystyle{\int_{\mathbb{T}^{n-m}}  \frac{\nabla_{x}\pi_{t}}{\pi_{t}}\cdot \frac{\pi_{t}}{\pi^{\xi}_{t}}\, \dd y}-\displaystyle{\int_{\mathbb{T}^{n-m}} \frac{\nabla_{x}\pi_{\infty}}{\pi_{\infty}}\cdot \frac{\pi_{t}}{\pi^{\xi}_{t}} \, \dd y} - \frac{\nabla_{x}\pi^{\xi}_{t}}{\pi^{\xi}_{t}}+\partial_{x} \ln (1)\\
&= \frac{\nabla_{x}\pi^{\xi}_{t}}{\pi^{\xi}_{t}}-\displaystyle{\int_{\mathbb{T}^{n-m}} \frac{\nabla_{x}\pi_{\infty}}{\pi_{\infty}}\cdot \frac{\pi_{t}}{\pi^{\xi}_{t}} \, \dd y} - \frac{\nabla_{x}\pi^{\xi}_{t}}{\pi^{\xi}_{t}}\\
&= -\displaystyle{\int_{\mathbb{T}^{n-m}} \nabla_{x}\left(-V(x,y)+A(x) \right) \cdot \frac{\pi_{t}}{\pi^{\xi}_{t}} \, \dd y}\\
&= \displaystyle{\int_{\mathbb{T}^{n-m}} \nabla_{x} V(x,y)  \cdot \frac{\pi_{t}}{\pi^{\xi}_{t}} \, \dd y} -\displaystyle{\int_{\mathbb{T}^{n-m}} \nabla A(x)  \cdot \frac{\pi_{t}}{\pi^{\xi}_{t}} \, \dd y} \\
&= G_{t}(x) -\nabla A(x).
\end{align*}

\end{proof}
In the following proofs, an integral over $\IT^n$ is with respect to $(x,y) \in \IT^m \times \IT^{n-m}$, an integral over $\IT^m$ is with respect to $x \in \IT^{m}$, and an integral over $\IT^{n-m}$ is with respect to $y \in \IT^{n-m}$.

\begin{proof}[Proof of Theorem~\ref{ThmConservative}]\

\noindent\textbf{Step 1:} Since for all $t\geq 0$, $E(t)=E_{m}(t)+E_{M}(t)$, using \eqref{dEM} and Lemma~\ref{lem:dEtot}, one has:
\begin{align*}
\frac{\dd E_{m}}{\dd t } &= - \displaystyle{\int_{\mathbb{T}^{n}} \left|\nabla \ln \po \frac{\pi_{t}}{ \pi_{\infty}}\pf\right|^{2} \, \pi_{t}}+\displaystyle{\int_{\mathbb{T}^{n}} \left(B_{t}- \nabla A \right) \cdot \nabla_{x} \ln \po \frac{\pi_{t}}{ \pi_{\infty}}\pf \,\pi_{t}} + \displaystyle{\int_{\mathbb{T}^{m}} \left|\nabla_{x} \ln \po \frac{\pi^{\xi}_{t}}{\pi_{\infty}^{\xi}}\pf\right|^{2} \, \pi^{\xi}_{t}}\\
&= - \displaystyle{\int_{\mathbb{T}^{n}} \left|\nabla \ln \po\frac{\pi_{t}}{ \pi_{\infty}}\pf\right|^{2} \, \pi_{t}}  + \displaystyle{\int_{\mathbb{T}^{n}} \left(G_{t}- \nabla A \right) \cdot \nabla_{x} \ln \po\frac{\pi_{t}}{ \pi_{\infty}}\pf \,\pi_{t}} + \displaystyle{\int_{\mathbb{T}^{m}} \left|\nabla_{x} \ln \po\frac{\pi^{\xi}_{t}}{\pi_{\infty}^{\xi}}\pf \right|^{2} \, \pi^{\xi}_{t}}+J_t
\end{align*}
where $$J_t:= \int_{\mathbb{T}^{n}} \left(B_{t}-G_{t} \right) \cdot \nabla_{x} \ln \po \frac{\pi_{t}}{ \pi_{\infty}}\pf \,\pi_{t} .$$ 
 Now, using Lemma~\ref{lem:ExpressionGt-B}, one gets:
\begin{align*}
\frac{\dd E_{m}}{\dd t } & =  - \displaystyle{\int_{\mathbb{T}^{n}} \left|\nabla \ln \po\frac{\pi_{t}}{ \pi_{\infty}}\pf\right|^{2} \, \pi_{t}} + \displaystyle{\int_{\mathbb{T}^{n}} \int_{\mathbb{T}^{n-m}} \nabla_{x} \ln \po\frac{\pi_{t}}{\pi_{\infty}}\pf \frac{\pi_{t}}{\pi^{\xi}_{t}} \, \dd y \cdot \nabla_{x} \ln \po \frac{\pi_{t}}{ \pi_{\infty}} \pf\,\pi_{t}} \\
&-\displaystyle{\int_{\mathbb{T}^{n}}\nabla_{x} \ln \po\frac{\pi^{\xi}_{t}}{\pi_{\infty}^{\xi}}\pf\cdot \nabla_{x} \ln \po \frac{\pi_{t}}{ \pi_{\infty}}\pf \,\pi_{t}} 
  +\displaystyle{\int_{\mathbb{T}^{m}} \left|\nabla_{x} \ln \po \frac{\pi^{\xi}_{t}}{\pi_{\infty}^{\xi}}\pf\right|^{2} \, \pi^{\xi}_{t}} + J_t.
\end{align*}
On the one hand, using Cauchy-Schwarz's inequality, the first terms in the right-hand side can be bounded as follows:
\begin{align*}
\displaystyle{\int_{\mathbb{T}^{n}} \int_{\mathbb{T}^{n-m}} \nabla_{x} \ln \po\frac{\pi_{t}}{\pi_{\infty}}\pf \frac{\pi_{t}}{\pi^{\xi}_{t}} \, \dd y \cdot \nabla_{x} \ln \po\frac{\pi_{t}}{ \pi_{\infty}}\pf \,\pi_{t}} &= \dst{\int_{\IT^{m}} \left|\int_{\IT^{n-m}} \nabla_{x} \ln \left(\frac{\pi_{t}}{\pi_{\infty}}\right) \pi_{t}\right|^{2} \frac{1}{\pi_{t}^{\xi}}}\\
&\leq \dst{\int_{\IT^{n}} \left|\nabla_{x} \ln \left(\frac{\pi_{t}}{\pi_{\infty}}\right)\right|^{2} \pi_{t}}.
\end{align*}
On the other hand, factorising the two next terms in the right-hand side, and using again Lemma~\ref{lem:ExpressionGt-B} gives:
\begin{align*}
\lefteqn{-\displaystyle{\int_{\mathbb{T}^{n}}\nabla_{x} \ln \po\frac{\pi^{\xi}_{t}}{\pi_{\infty}^{\xi}}\pf\cdot \nabla_{x} \ln \po\frac{\pi_{t}}{ \pi_{\infty}}\pf \,\pi_{t}} +\displaystyle{\int_{\mathbb{T}^{m}} \left|\nabla_{x} \ln \po\frac{\pi^{\xi}_{t}}{\pi_{\infty}^{\xi}}\pf\right|^{2} \, \pi^{\xi}_{t}}}\\
 &= \int_{\mathbb{T}^{m}} \nabla_{x} \ln \po\frac{\pi^{\xi}_{t}}{\pi_{\infty}^{\xi}}\pf \cdot\left(\nabla_{x} \ln \po\frac{\pi^{\xi}_{t}}{\pi_{\infty}^{\xi}}\pf  -\int_{\mathbb{T}^{n-m}} \nabla_{x} \ln \po\frac{\pi_{t}}{\pi_{\infty}}\pf \frac{\pi_{t}}{\pi^{\xi}_{t}}  \right) \pi^{\xi}_{t} \\
&= \displaystyle{\int_{\mathbb{T}^{m}} \nabla_{x} \ln \po\frac{\pi^{\xi}_{t}}{\pi_{\infty}^{\xi}}\pf (\nabla A-G_{t}) \pi^{\xi}_{t} }.
\end{align*}
Using once again the Cauchy-Schwarz's inequality, one gets:
\begin{align*}
\frac{\dd E_{m}}{\dd t } &\leq  - \displaystyle{\int_{\mathbb{T}^{n}} \left|\nabla_{y} \ln \po\frac{\pi_{t}}{ \pi_{\infty}}\pf\right|^{2} \, \pi_{t}} + \left( \displaystyle{\int_{\mathbb{T}^{m}}  \left|\nabla A-G_{t}\right|^{2} \pi^{\xi}_{t} }\right)^{\frac{1}{2}} \left( \displaystyle{\int_{\mathbb{T}^{m}}\left|\nabla_{x} \ln \po\frac{\pi^{\xi}_{t}}{\pi_{\infty}^{\xi}}\pf\right|^{2} \pi^{\xi}_{t} }\right)^{\frac{1}{2}}+J_t.
\end{align*}
Eventually, recalling that $ \int_{\mathbb{T}^{m}}|\nabla_{x} \ln \po\frac{\pi^{\xi}_{t}}{\pi_{\infty}^{\xi}}\pf|^{2} \pi^{\xi}_{t} =I_{M}(t)$ and relying on \eqref{BoundEm} and Lemma~\ref{lem:BoundBias}, one has, for all $\varepsilon>0$ and all $t \geq 0$:
\begin{align}
\frac{\dd E_{m}}{\dd t } &\leq -2\rho E_{m}(t) + M \sqrt{\frac{2}{\rho}} \sqrt{E_{m}(t)} \sqrt{I_{M}(t)}  + J_t \nonumber\\
&\leq  -2\rho E_{m}(t) +  2\sqrt{\varepsilon \rho E_{m}(t)} \sqrt{\frac{M^{2}}{2\rho^{2} \varepsilon}I_{M}(t)}  + J_t \nonumber \\
&\leq -(2-\varepsilon) \rho E_{m}(t) + \frac{M^{2}}{2\rho^{2}\varepsilon} I_{M}(t)+J_t. \label{DemThmConservativeBoundEm}
\end{align}

\noindent \textbf{Step 2:} In order to set the idea of the proof, let us first treat the case of the ABF algorithm, where one simply has $J_t=0$ for all $t \ge 0$. Inequality~\eqref{DemThmConservativeBoundEm} yields, for all $\varepsilon>0$ and all $ t \ge 0$:
\begin{align*}
\frac{\dd E_{m}}{\dd t } &\leq -(2-\varepsilon) \rho E_{m}(t) + \frac{M^{2}}{2\rho^{2}\varepsilon} I_{M}(t),
\end{align*}
and using Gronwall's lemma, one has,  for all $\varepsilon>0$ for all $t \geq 0$, :
\begin{align*}
E_{m}(t) &\leq E_{m}(0)e^{-(2-\varepsilon)\rho t} + \frac{M^2}{2\rho^2 \varepsilon} \dst{\int_{0}^{t} I_{M}(s) e^{-(2-\varepsilon)\rho (t-s)} \, \dd s}.
\end{align*}
\begin{rem} Note that in the ABF case \cite[Lemma 12]{LRS07} or the PABF case with a Helmholtz projection done with respect to the marginal density $\pi_{t}^{\xi}$ \cite[Corollary 1]{Alrachid}, one has the exponential convergence towards zero of the macroscopic Fisher information $I_{M}(t)$. This is not the case when one considers the classical Helmholtz projection with respect to the Lebesgue measure: indeed, the density $\pi_{t}^{\xi}$ does not satisfy the heat equation anymore, but an elliptic equation~\eqref{FP-Pi_xi} with a null-divergence drift. Having no additional information about the regularity of the drift, one cannot prove the convergence of $I_{M}(t)$ towards zero in the long-time limit as done in \cite{Alrachid,LRS07}.
\end{rem}
By Proposition~\ref{PropHistoPlat}, for all $t\geq 0$, $E_{M}(t) \leq E_{M}(0)e^{-8\pi^{2}t}$. Since $I_{M}(t)=-E'_{M}(t)$, one gets:
\begin{equation}
0\leq F(t):=\dst{\int_{t}^{\infty} I_{M}(s) \, \dd s } \leq E_{M}(t)\leq E_{M}(0)e^{-8\pi^{2}t}, \qquad \forall t \geq 0. \label{DemThmConservativeBoundFisher}
\end{equation} 
Consequently, relying on~\eqref{DemThmConservativeBoundFisher} one has, for all $\varepsilon>0$, for all $t \geq 0$:
\begin{align*}
\dst{\int_{0}^{t} I_{M}(s) e^{-(2-\varepsilon)\rho (t-s)} \, \dd s} & = - e^{-(2-\varepsilon) \rho t} \dst{\int_{0}^{t} F'(s) e^{(2-\varepsilon)\rho s} \, \dd s}\\
&=e^{-(2-\varepsilon) \rho t} \left(\dst{\int_{0}^{t} F(s) (2-\varepsilon)\rho e^{(2-\varepsilon)\rho s} \, \dd s} - \left[ F(s)e^{(2-\varepsilon)\rho s}\right]_{0}^{t}\right)\\
&\leq e^{-(2-\varepsilon) \rho t} \left( (2-\varepsilon)\rho E_{M}(0) \dst{\int_{0}^{t} e^{-(8\pi^{2}-(2-\varepsilon)\rho)s} \, \dd s } -F(t)e^{(2-\varepsilon)\rho t} +F(0)\right)\\
&\leq e^{-(2-\varepsilon) \rho t} \left((2-\varepsilon)\rho E_{M}(0) \dst{\int_{0}^{t} e^{-(8\pi^{2}-(2-\varepsilon)\rho)s} \, \dd s } -F(t)e^{(2-\varepsilon)\rho t} +E_{M}(0)\right)\\
&\leq E_{M}(0)e^{-(2-\varepsilon) \rho t}\left((2-\varepsilon)\rho  \dst{\int_{0}^{t} e^{-(8\pi^{2}-(2-\varepsilon)\rho)s} \, \dd s } +1\right)
\end{align*}
We distinguish between two case:
\begin{itemize}
\item[(i)] If $8\pi^{2} = (2-\varepsilon) \rho$, one gets:
\begin{align*}
\dst{\int_{0}^{t} I_{M}(s) e^{-(2-\varepsilon)\rho (t-s)} \, \dd s} &\leq E_{M}(0) e^{-8\pi^{2}t} \left( 8\pi^2 \, t+1\right).
\end{align*}
Since for all $\delta >0$ and all $t \geq 0$, one has $t \leq \frac{e^{-1}}{\delta} e^{\delta t}$, choosing $\delta=\varepsilon$ yields, for all $t\geq 0$:
\begin{align*}
\dst{\int_{0}^{t} I_{M}(s) e^{-(2-\varepsilon)\rho (t-s)} \, \dd s} &\leq E_{M}(0)\left(\frac{8\pi^2}{e\,\varepsilon} \vee 1 \right) e^{-(8\pi^{2}-\varepsilon)t}.
\end{align*}
\item[(ii)] If $8\pi^{2} \neq (2-\varepsilon)\rho$, one gets, in all cases ($8\pi^{2} > (2-\varepsilon)\rho$ or $8\pi^{2}<(2-\varepsilon)\rho$):
\begin{align*}
\dst{\int_{0}^{t} I_{M}(s) e^{-(2-\varepsilon)\rho (t-s)} \, \dd s} &\leq E_{M}(0)\left( \frac{(2-\varepsilon)\rho}{|8\pi^{2}-(2-\varepsilon)\rho|} \vee 1 \right)  e^{-(8\pi^{2} \wedge (2-\varepsilon)\rho)t}.
\end{align*}
\end{itemize} Which yields, 
\begin{align*}
E_{m}(t) &\leq E_{m}(0)e^{-(2-\varepsilon)\rho t} + \frac{M^2}{2\rho^2 \varepsilon} \dst{\int_{0}^{t} I_{M}(s) e^{-(2-\varepsilon)\rho (t-s)} \, \dd s}\\
&\leq \left(E_{m}(0) \vee \frac{M^2}{2\rho^2 \varepsilon}E_{M}(0)   \left( \frac{8\pi^2}{e \varepsilon} \vee \frac{(2-\varepsilon)\rho}{|8\pi^{2} - (2-\varepsilon)\rho|} \vee 1\right) \right)e^{-((8\pi^{2}-\varepsilon)\wedge (2-\varepsilon)\rho)t}.
\end{align*}
\noindent\textbf{Conclusion:}  for the ABF algorithm, we have obtained that for all $\varepsilon>0$, there exists $\mathcal{K}=\mathcal{K}(\varepsilon)>0$ such that for all $t \ge 0$,
$$E_{m}(t) \leq \mathcal{K} e^{-\left((8\pi^{2} \wedge 2\rho)-\varepsilon \right) t},$$
where $\dst{\mathcal{K} =  \left(E_{m}(0) \vee \frac{M^2}{2\rho \varepsilon}E_{M}(0)   \left( \frac{8\pi^2 \rho}{e \varepsilon} \vee \frac{(2\rho-\varepsilon)}{|8\pi^{2} - (2\rho-\varepsilon)|} \vee 1\right) \right)}$. \\

\noindent \textbf{Step 3:} Let us now concentrate on the PABF case, and let us prove an upper bound on $J_t$. For $t\geq 0$, recall the notation $\nabla H_{t}:=\mathsf{P}_{L^{2}(\lambda)}(G_{t})$, so that $B_t = \nabla H_t$. Similarly, let us introduce, for all $t\geq 0$,
$$\nabla \tilde{H}_{t}:= \mathsf{P}_{L^{2}(\pi_{t}^{\xi})}(G_{t}).$$
Recall that $\mathsf{P}_{L^{2}(\nu)}(f)$ stands for the Helmholtz projection of a vector field $f$ with respect to the measure $\nu$.
In the conservative case one has $\pi_{\infty} \varpropto e^{-V+A}$, so that $G_\infty=\nabla A$. Since   $G_\infty$ is a gradient, one has:
$$  \nabla H_\infty= \mathsf{P}_{L^{2}(\lambda)}(G_\infty) = \, \nabla A \, = \mathsf{P}_{L^{2}(\pi_{\infty}^{\xi})}(G_\infty)=\nabla \tilde{H_\infty}  .$$ 
On the contrary, there is no reason  for $\nabla H_{t}$ and $\nabla \tilde{H}_{t}$ to be equal at a fixed time $t>0$. Let us decompose
\begin{align*}
J_t&= \displaystyle{\int_{\mathbb{T}^{n}} (\nabla H_{t}- \nabla \tilde{H}_{t})\cdot  \nabla_{x} \ln\po \frac{\pi_{t}}{ \pi_{\infty}}\pf \,\pi_{t}}+\displaystyle{\int_{\mathbb{T}^{n}} (\nabla \tilde{H}_{t}- G_{t})\cdot \nabla_{x} \ln\po \frac{\pi_{t}}{ \pi_{\infty}}\pf \,\pi_{t}}.
\end{align*}
As proven in \cite[Lemma 6]{Alrachid}, relying on the fact that since $\nabla_{x}\ln \left(\pi_{\infty}\right)= -\nabla \left(V-A\right)$,~$\pi^{\xi}_{\infty} \equiv 1$ and $\nabla \tilde{H}_{t} =\mathsf{P}_{L^{2}(\pi_{t}^{\xi})}(G_t)$, one can show that the last right-hand term is negative. One consequently has:
\begin{align*}
\displaystyle{\int_{\mathbb{T}^{n}} (\nabla H_{t}- \nabla \tilde{H}_{t}) \nabla_{x} \ln\po \frac{\pi_{t}}{ \pi_{\infty}}\pf \,\pi_{t}}&= \displaystyle{\int_{\mathbb{T}^{n}}(\nabla H_{t}- \nabla \tilde{H}_{t})\cdot \nabla_{x} \ln(\pi_{t}) \, \pi_{t}} -\displaystyle{\int_{\mathbb{T}^{n}}(\nabla H_{t}- \nabla \tilde{H}_{t}) \cdot\nabla_{x} \ln(\pi_{\infty}) \,\pi_{t}}\\
&= \displaystyle{\int_{\mathbb{T}^{m}}(\nabla H_{t}- \nabla \tilde{H}_{t}) \cdot \nabla_{x} \pi_{t}^{\xi}} -\displaystyle{\int_{\mathbb{T}^{n}}(\nabla H_{t}- \nabla \tilde{H}_{t}) \cdot \nabla_{x} \ln(\pi_{\infty}) \,\pi_{t}}\\
&=\displaystyle{\int_{\mathbb{T}^{m}}(\nabla H_{t}- \nabla \tilde{H}_{t})\cdot \nabla_{x} \ln\po\frac{\pi^{\xi}_{t}}{ \pi_{\infty}^{\xi}}\pf \, \pi_{t}^{\xi}} -\displaystyle{\int_{\mathbb{T}^{n}}(\nabla H_{t}- \nabla \tilde{H}_{t})\cdot \nabla_{x} \ln(\pi_{\infty}) \,\pi_{t}}.
\end{align*}
Hence, in the PABF case,
\begin{align}
J_t&\leq    \displaystyle{\int_{\mathbb{T}^{m}} (\nabla H_{t}- \nabla \tilde{H}_{t}) \nabla_{x} \ln \po\frac{\pi^{\xi}_{t}}{ \pi_{\infty}^{\xi}}\pf \,\pi^{\xi}_{t}}- \displaystyle{\int_{\mathbb{T}^{n}} (\nabla H_{t}- \nabla \tilde{H}_{t}) \nabla_{x} \ln(  \pi_{\infty}) \,\pi_{t}}\nonumber\\
& =  \displaystyle{\int_{\mathbb{T}^{m}} (\nabla H_{t}- \nabla \tilde{H}_{t}) \nabla_{x} \ln\po\frac{\pi^{\xi}_{t}}{ \pi_{\infty}^{\xi}}\pf \,\pi^{\xi}_{t}}-\displaystyle{\int_{\mathbb{T}^{m}} (\nabla H_{t}- \nabla \tilde{H}_{t}) (\nabla A-G_{t}) \,\pi^{\xi}_{t}} \nonumber\\
&\leq  \left(\displaystyle{\int_{\mathbb{T}^{m}} |\nabla H_{t}- \nabla \tilde{H}_{t}|^{2} \, \pi^{\xi}_{t} } \right) ^{\frac{1}{2}}  \left( \left(\displaystyle{\int_{\mathbb{T}^{m}} \left| \nabla_{x} \ln\po \frac{\pi^{\xi}_{t}}{ \pi_{\infty}^{\xi}}\pf\right|^{2} \, \pi^{\xi}_{t}}\right) ^{\frac{1}{2}}+\left( \displaystyle{\int_{\mathbb{T}^{m}} | \nabla A-G_{t}|^{2} \, \pi^{\xi}_{t} }\right) ^{\frac{1}{2}} \right) \nonumber\\
&\leq  \left(\dst{\int_{\mathbb{T}^{m}} |\nabla H_{t}- \nabla \tilde{H}_{t}|^{2} \, \pi^{\xi}_{t} } \right) ^{\frac{1}{2}}  \left(\sqrt{I_{M}(t)} + M \sqrt{\frac{2}{\rho}} \sqrt{E_{m}(t)} \right). \label{DemThmConservativeBoundJ}
\end{align}

\noindent\textbf{Step 4:} We will now consider times such that $t \geq 1$. Since $\nabla \tilde{H}_{t} =\mathsf{P}_{L^{2}(\pi_{t}^{\xi})}(G_t)$, one has:
$$\displaystyle{\int_{\mathbb{T}^{m}} |\nabla H_{t}-G_{t}|^{2}\pi_{t}^{\xi} }= \displaystyle{\int_{\mathbb{T}^{m}} |\nabla \tilde{H}_{t}-G_{t}|^{2}\pi_{t}^{\xi}  }+\displaystyle{\int_{\mathbb{T}^{m}} |\nabla \tilde{H}_{t}-\nabla H_{t}|^{2}\pi_{t}^{\xi}  },$$
which yields:
\begin{align*}
\displaystyle{\int_{\mathbb{T}^{m}} |\nabla H_{t}- \nabla \tilde{H}_{t}|^{2} \, \pi_{t}^{\xi} } &=\dst{\int_{\IT^m}  |\nabla H_t -G_t|^2 \pi_{t}^{\xi}}-\dst{\int_{\IT^m} |\nabla \tilde{H}_t -G_t|^2  \pi_{t}^{\xi}}\\
&\leq \|\pi_{t}^{\xi}\|_{\infty} \dst{\int_{\IT^m}  |\nabla H_t -G_t|^2 }-\dst{\int_{\IT^m} |\nabla \tilde{H}_t -G_t|^2  \pi_{t}^{\xi}}\\
&\leq  \|\pi_{t}^{\xi}\|_{\infty} \dst{\int_{\IT^m}  |\nabla \tilde{H}_t -G_t|^2 }-\dst{\int_{\IT^m} |\nabla \tilde{H}_t -G_t|^2  \pi_{t}^{\xi}}\\
&\leq \|\pi_{t}^{\xi}\|_{\infty} \left(\left\| 1/\pi^{\xi}_{t} \right\|_{\infty} -1\right) \dst{\int_{\IT^m}  |\nabla \tilde{H}_t -G_t|^2 \pi_{t}^{\xi} }\\
&\leq  \|\pi_{t}^{\xi}\|_{\infty} \left(\left\| 1/\pi^{\xi}_{t} \right\|_{\infty} -1\right) \dst{\int_{\IT^m}  |G_t|^2 \pi_{t}^{\xi} }\\
& \leq  \|\pi_{t}^{\xi}\|_{\infty} \left(\left\|1/\pi^{\xi}_{t} \right\|_{\infty} -1\right) M^2,
\end{align*}
where we used that, under Assumption~\ref{HypoGeneral}, $\|G_{t}\|_{\infty} \leq \|\nabla_{x}V \|_{\infty} \leq M$. Now, from Proposition~\ref{PropHistoUniform}, there exists $C \geq 0$ such that, for all $t\geq 1$:
\[\|\pi^{\xi}_{t}\|_{\infty} \leq 1+Ce^{-4\pi^{2}t}\qquad\text{and}\qquad 
\|\displaystyle{1/\pi^{\xi}_{t}}\|_{\infty} \leq 1+Ce^{-4\pi^{2}t}.\]
This yields the existence of a constant $\tilde{C}>0$ such that, for all $t \geq 1$:
$$\left(\displaystyle{\int_{\mathbb{T}^{m}} |\nabla H_{t}- \nabla \tilde{H}_{t}|^{2} \, \pi^{\xi}_{t} }\right)^{\frac{1}{2}}\leq  \tilde{C} e^{-2\pi^{2}t},$$
and, for all $\varepsilon>0$, for all $t \geq 1$:
\begin{align*}
J_t&\leq   \tilde{C} e^{-2\pi^{2}t} \left( \sqrt{I_{M}}(t) + M \sqrt{\frac{2}{\rho}} \sqrt{E_{m}(t)}\right)\\
&\leq   \tilde{C} e^{-2\pi^{2}t} \left( \sqrt{I_{M}}(t) +  2\sqrt{\frac{M^2}{2\rho^{2} \varepsilon}} \sqrt{ \varepsilon \rho E_{m}(t)}\right)\\
&\leq \varepsilon \rho E_{m}(t) + I_{M}(t) +\left(\frac{\tilde{C}^2}{4}+\frac{M^{2}\tilde{C}^2}{2\rho^{2}\varepsilon}\right)e^{-4\pi^{2}t}.
\end{align*}
Hence one gets:
$$\frac{\dd E_{m}}{\dd t } \leq  -(2-2\varepsilon)\rho E_{m}(t) + K_{1}I_{M}(t)+K_{2}e^{-4\pi^{2}t}, \quad \forall t \geq 1,$$
with $$ K_{1}=K_{1}(\varepsilon)=1+\frac{M^{2}}{2\rho^{2} \varepsilon}, \qquad
K_{2}=K_{2}(\varepsilon)=\frac{\tilde{C}^2}{4}+\frac{M^{2}\tilde{C}^2}{2\rho^{2} \varepsilon}.$$

From now on, let us fix $\varepsilon\in (0,1)$ and denote by $r_{\varepsilon}:= 2(1-\varepsilon)$. Using Gronwall's lemma yields, for all $t \geq 1$:
$$ E_{m}(t) \leq E_{m}(1)e^{r_{\varepsilon}\rho} \, e^{-r_{\varepsilon}\rho t} +\dst{\int_{1}^{t} K_{1}I_{M}(s)e^{-r_{\varepsilon}\rho (t-s)} +K_{2}e^{-4\pi^{2} s -r_{\varepsilon}\rho(t-s)} \, \dd s}.$$
$\bullet$ Let us first consider, for all $t \geq 1$, $I_{1}:=K_{1}\dst{\int_{1}^{t} I_{M}(s)e^{-r_{\varepsilon}\rho (t-s)}}$. As done in Step 2, relying on \eqref{DemThmConservativeBoundFisher}, one has, for all $t \ge 1$:
\begin{align*}
I_{1} &=  -K_{1} e^{-r_{\varepsilon} \rho t} \dst{\int_{1}^{t} F'(s) e^{r_{\varepsilon}\rho s} \, \dd s}\\
&= K_{1} e^{-r_{\varepsilon} \rho t} \left(r_{\varepsilon}\rho \dst{\int_{1}^{t} F(s) e^{r_{\varepsilon} s} \, \dd s} -F(t)e^{r_{\varepsilon}\rho t} +F(1)e^{r_{\varepsilon}\rho}\right)\\
&\leq K_{1}E_{M}(0)e^{-r_{\varepsilon} \rho t}  \left(r_{\varepsilon}\rho  \dst{\int_{1}^{t}  e^{-(8\pi^{2}-r_{\varepsilon}) s} \, \dd s}+e^{-(8\pi^{2}-r_{\varepsilon}\rho)}\right).
\end{align*}
We distinguish between two cases:
\begin{itemize}
\item[(i)] If $8\pi^{2} = r_{\varepsilon} \rho$, one gets, for all $t\geq 1$:
\begin{align*}
I_1 &\leq K_{1}E_{M}(0)e^{-8\pi^2 t}  \left(8\pi^2  \left(t-1\right)+1\right)
\end{align*}
and, since $(t-1) \leq \frac{e^{-1-\delta}}{\delta}e^{\delta t}$ for all $\delta>0$, considering $\delta = \varepsilon$, one gets that, for all $t \geq 1$:
\begin{align*}
I_1 &\leq K_{1}E_{M}(0)e^{-8\pi^2 t}  \left(8\pi^2\frac{e^{-1-\varepsilon}}{\varepsilon}   e^{\varepsilon t}+1\right)\\
&\leq K_{1} E_{M}(0)\left(8\pi^2 \frac{ e^{-1-\varepsilon}}{\varepsilon}  \vee 1 \right) e^{-(8\pi^{2}-\varepsilon) t}.
\end{align*}
\item[(ii)] If $8\pi^{2} \neq r_{\varepsilon} \rho$, one gets, for all $t\geq 1$:
\begin{align*}
I_1 & \leq K_{1} E_{M}(0)\left(\frac{r_{\varepsilon}\rho}{|8\pi^{2}-r_{\varepsilon}\rho |}  \vee e^{-(8\pi^{2}-r_{\varepsilon}\rho)} \right) e^{-(8\pi^{2} \wedge r_{\varepsilon} \rho) t}.
\end{align*}
\end{itemize}
In any case one has, for all $t\geq 1$
$$I_{1} \leq \mathcal{K}_{1} e^{-\left((8\pi^2 - \varepsilon) \wedge r_{\varepsilon} \rho\right) t},$$
where $\dst{\mathcal{K}_{1}=\mathcal{K}_{1}(\varepsilon)=\left(1+\frac{M^{2}}{2\rho^{2} \varepsilon}\right)E_{M}(0) \left( 8\pi^2\frac{e^{-1-\varepsilon}}{\varepsilon} \vee \frac{r_{\varepsilon}\rho}{|8\pi^{2}-r_{\varepsilon}\rho|}\vee e^{-(8\pi^{2}-r_{\varepsilon}\rho)} \vee 1\right)>0}$.\\

\noindent $\bullet$ Now consider, for all $t \ge 1$, $I_2 := K_{2}\dst{\int_{1}^{t}e^{-4\pi^{2} s -r_{\varepsilon}\rho(t-s)} \, \dd s}$. We distinguish between two cases:
\begin{itemize}
\item[(i)] If $r_{\varepsilon} \rho \neq 4 \pi^2$ then, for all $t\geq 1$:
\begin{align*}
K_{2}\dst{\int_{1}^{t}e^{-4\pi^{2} s -r_{\varepsilon}\rho(t-s)} \, \dd s} &\leq \frac{K_2}{|4\pi^{2}-r_{\varepsilon} \rho |} e^{-\left(4\pi^{2} \wedge r_{\varepsilon}\rho \right) t}. 
\end{align*}
\item[(ii)] If $r_{\varepsilon} \rho = 4 \pi^2$ then, for all $t\geq 1$:
\begin{align*}
K_{2}\dst{\int_{1}^{t}e^{-4\pi^{2} s -r_{\varepsilon}\rho(t-s)} \, \dd s} &= K_{2}e^{-4\pi^{2}t} (t-1),
\end{align*}
and, since $(t-1) \leq \frac{e^{-1-\delta}}{\delta}e^{\delta t}$ for all $\delta>0$, considering $\delta = \varepsilon$, one gets that, for all $t \geq 1$:
\begin{align*}
K_{2}\dst{\int_{1}^{t}e^{-4\pi^{2} s -r_{\varepsilon}\rho(t-s)} \, \dd s} & \leq K_2  \frac{e^{-1-\varepsilon}}{\varepsilon} e^{-\left(4\pi^{2}-\varepsilon\right)t}.
\end{align*}
\end{itemize}
In any case one has, for all $t\geq 1$:
$$I_{2} \leq \mathcal{K}_2 e^{-\left((4\pi^{2}-\varepsilon )\wedge r_{\varepsilon}\rho\right)t},$$
where $\dst{\mathcal{K}_2= \mathcal{K}_{2}(\varepsilon)= \left(\frac{\tilde{C}^2}{4}+\frac{M^{2}\tilde{C}^2}{2\rho^{2} \varepsilon}\right) \left(\frac{1}{|4\pi^{2}-r_{\varepsilon} \rho |} \vee \frac{e^{-1-\varepsilon}}{\varepsilon} \right)>0.}$\\

\noindent Hence, recalling that $r_{\varepsilon}=2(1-\varepsilon)$ one gets that for all $\varepsilon>0$, for all $t \ge 1$,
\begin{align*}
  E_{m}(t) &  \leq  E_{m}(1)e^{r_{\varepsilon}\rho}\, e^{-r_{\varepsilon}\rho t} +\mathcal{K}_{1}e^{-\left((8\pi^2 - \varepsilon) \wedge r_{\varepsilon} \rho\right) t} + \mathcal{K}_{2}e^{-\left((4\pi^{2}-\varepsilon )\wedge r_{\varepsilon}\rho\right)t} \\
  &\leq \mathcal{K}_{3} e^{-\left((4\pi^{2} \wedge 2 \rho)-\varepsilon\right) t},
\end{align*} for some $\mathcal{K}_{3}=\mathcal{K}_{3}(\varepsilon)= \left(E_{m}(1)e^{2\rho-\varepsilon} \vee \tilde{\mathcal{K}}_{1} \vee \tilde{\mathcal{K}}_{2}\right)>0$, where
$$\left\{\begin{array}{l}
\tilde{\mathcal{K}}_{1}=\dst{\left(1+\frac{M^{2}}{\rho \varepsilon}\right)E_{M}(0) \left( \frac{16\pi^{2}\rho\,\,e^{-(1+\frac{\varepsilon}{2\rho})} }{\varepsilon} \vee \frac{(2\rho -\varepsilon) }{\left|8\pi^{2}-(2\rho-\varepsilon)\right|} \vee e^{-(8\pi^{2}-(2\rho-\varepsilon))} \vee 1\right)}\\
\tilde{\mathcal{K}}_{2}=\dst{\left(\frac{\tilde{C}^2}{4}+\frac{M^{2}\tilde{C}^2}{\rho \varepsilon}\right) \left(\frac{1}{|4\pi^{2}-(2\rho-\varepsilon) |} \vee \frac{2\rho e^{-(1+\frac{\varepsilon}{2\rho})}}{\varepsilon} \right)}
\end{array}
\right.\,.$$\\

\noindent\textbf{Step 5:} It remains to treat the case where $t \in [0,1]$. We have:
$$ \left(\dst{\int_{\mathbb{T}^{m}} |\nabla H_{t}- \nabla \tilde{H}_{t}|^{2} \, \pi^{\xi}_{t} } \right) ^{\frac{1}{2}} \leq \|\pi_{t}^{\xi}\|_{L^{2}(\IT^{m})}^{\frac{1}{2}} \|\nabla H_{t} -\nabla \tilde{H}_{t} \|_{L^{4}(\IT^{m})}, \quad \forall t \in [0,1].$$
From \eqref{eq:pixiL2}, there exists $C_{2}>0$ such that for all $t\in [0,1]$, $\|\pi_{t}^{\xi}\|_{L^{2}(\IT^{m})}^{\frac{1}{2}}\le C_{2}$, and, using ~\cite[Lemma 15.13]{Ambrosio}, there exists $C_{4}>0$ such that for all $t \in [0,1]$, 
$$\|\nabla H_{t}\|_{L^{4}(\IT^{m})}\leq C_{4} \|G_{t}\|_{L^{4}(\IT^{m})} \leq C_{4} \|\mathcal{F}\|_{\infty} \leq C_{4} \|\nabla V\|_{\infty} <\infty.$$
Similarly, one has $\|\nabla \tilde{H}_{t}\|_{L^{4}(\IT^{m})} \leq C_{4} \|\nabla V\|_{\infty}$. Hence inequality \eqref{DemThmConservativeBoundJ} becomes, for all $\varepsilon>0$ and for all $t\in [0,1]$:
\begin{align*}
J_t&\leq  2C_{2}C_{4} \|\nabla V\|_{\infty} \left(\sqrt{I_{M}}(t) + M \sqrt{\frac{2}{\rho}} \sqrt{E_{m}(t)} \right)\\
&\leq  2C_{2}C_{4} \|\nabla V\|_{\infty} \left(\sqrt{I_{M}}(t) +  \sqrt{\frac{2M^{2}}{\varepsilon \rho^{2} }} \sqrt{\varepsilon \rho E_{m}(t)} \right)\\
&\leq  \varepsilon \rho E_{m}(t) + I_{M}(t) +(C_{2}C_{4} \|\nabla V\|_{\infty})^{2}\left(1+\frac{2M^{2}}{\varepsilon\rho^{2}}\right).
\end{align*}
It yields, from inequality \eqref{DemThmConservativeBoundEm}, for all $\varepsilon>0$ and for all $t \in [0,1[$:
$$
\frac{\dd E_{m}}{\dd t } \leq  -r_{\varepsilon}\rho E_{m}(t)+K_{1}I_{M}(t)+K_{2},$$
with $$
K_{1}=K_{1}(\varepsilon)=1+\frac{M^{2}}{2\varepsilon \rho^{2}} ,\qquad 
K_{2}=K_{2}(\varepsilon)=(C_{2}C_{4} \|\nabla V \|_{\infty})^{2}\left(1+\frac{2M^{2}}{\varepsilon\rho^{2}}\right).
$$
The Gronwall's lemma yields, for all $\varepsilon>0$ and for all $t\in [0,1[$:
\begin{align*}
E_{m}(t) &\leq  E_{m}(0) e^{-r_{\varepsilon}\rho t} +K_{1}\dst{\int_{0}^{t} I_{M}(s) e^{-r_{\varepsilon}\rho(t-s)} \, \dd s } + K_{2} \dst{\int_{0}^{t} e^{-r_{\varepsilon}\rho(t-s)} \, \dd s}\\
&\leq E_{m}(0) +K_{1}e^{0}\dst{\int_{0}^{\infty} I_{M}(s)  \, \dd s } + \frac{K_{2}}{r_{\varepsilon}\rho}\left(1- e^{-r_{\varepsilon} \rho t}\right) \\
& \leq  E_m(0) + K_1 E_M(0) + \frac{K_2}{r_{\varepsilon} \rho},
\end{align*}
where we used \eqref{DemThmConservativeBoundFisher}. Hence, for all $\varepsilon>0$ and for all $t \in [0,1[$
$$E_m(t) e^{\left((4\pi^{2} \wedge 2\rho)-\varepsilon\right)} \leq \left( E_m(0) + K_1 E_M(0) + \frac{K_2}{r_{\varepsilon} \rho}\right) e^{\left((4\pi^{2} \wedge 2\rho)-\varepsilon\right)} <+\infty . $$

\noindent\textbf{Conclusion:} for the PABF algorithm, we have obtained that for all $\varepsilon>0$, there exists $C=C(\varepsilon)>0$ such that, for all $t\geq 0$, 
$$E_{m}(t) \leq \ C e^{-\left((4\pi^{2} \wedge 2\rho)-\varepsilon\right)t}.$$
Recall that by Proposition~\ref{PropHistoPlat}, $E_{M}(t) \leq E_{M}(0)e^{-8\pi^{2}t}$ for all $t \ge 0$. The decomposition $E(t) =E_{m}(t)+E_{M}(t)$ concludes the proof. 
\end{proof}

\subsection{Proof of Theorem \ref{ThmCVtempslong}}
\label{subsec:thm-non-conservative}

Let us prove Theorem~\ref{ThmCVtempslong}. 

\begin{proof}

Using Lemma~\ref{lem:dEtot} one gets:
\begin{align}
\frac{\dd E}{\dd t }&=- \displaystyle{\int_{\mathbb{T}^{n}} \left|\nabla \ln \po\frac{\pi_{t}}{ \pi_{\infty}}\pf\right|^{2} \, \pi_{t}}+\displaystyle{\int_{\mathbb{T}^{n}} \left(B_{t}- B\right) \cdot \nabla_{x} \ln \po\frac{\pi_{t}}{ \pi_{\infty}}\pf \,\pi_{t}}\nonumber \\
& \leq   -\displaystyle{\int_{\mathbb{T}^{n}} \left|\nabla \ln \po\frac{\pi_{t}}{ \pi_{\infty}}\pf\right|^{2} \, \pi_{t}} +\left(\displaystyle{\int_{\mathbb{T}^{m}} |B_{t}- B|^{2} \, \pi^{\xi}_{t}}\right)^{\frac{1}{2}}\left(  \displaystyle{\int_{\mathbb{T}^{n}} \left|\nabla_{x} \ln\po\frac{\pi_{t}}{\pi_{\infty}}\pf\right|^{2} \, \pi_{t}} \right)^{\frac{1}{2}}\,. \label{DemThmCvtempslongBoundE}
\end{align}
\noindent\textbf{Step 1:} Let us first consider $t \ge 1$. In the PABF case, since an orthogonal projection contracts the corresponding norm, for all $t \ge 1$:
\begin{align*}
\int_{\mathbb{T}^{m}} |\nabla H_{t}- \nabla H_{\infty}|^{2} \pi^{\xi}_{t} & \leq  \| \pi_t^{\xi}\|_\infty \int_{\mathbb{T}^{m}} |\nabla H_{t}- \nabla H_{\infty}|^{2}   \\
& \leq \| \pi_t^{\xi}\|_\infty \int_{\mathbb{T}^{m}} |G_{t}- G_{\infty}|^{2}     \\
& \leq \| \pi_t^{\xi}\|_\infty \|1/\pi_t^{\xi}\|_\infty \int_{\mathbb{T}^{m}} |G_{t}- G_{\infty}|^{2} \pi_t^{\xi}   \\
& \leq \po 1 + C e^{-4\pi^{2}t}\pf \int_{\mathbb{T}^{m}} |G_{t}- G_{\infty}|^{2} \pi_t^{\xi} ,
\end{align*}
for some $C>0$ according to Proposition~\ref{PropHistoUniform}. Together with Lemma~\ref{lem:BoundBias} and the microscopic log-Sobolev inequality \eqref{BoundEm}, we have thus obtained for all $t \ge 1$, in both the ABF case (where $B_t=G_t$ and $B_\infty= G_\infty$) and PABF case (where $B_{t}= \nabla H_t$ and $B_\infty = \nabla H_\infty$),
\begin{align*}
\left(\int_{\mathbb{T}^{m}} |B_{t}- B_{\infty}|^{2} \pi^{\xi}_{t}\right)^{\frac{1}{2}} &\leq  \sqrt{1 + C e^{-4\pi^{2}t}} M \sqrt{\frac{2}{\rho}} \sqrt{E_{m}(t)} \\
&\leq  \sqrt{1 + C e^{-4\pi^{2}t}} M \sqrt{\frac{2}{\rho}} \frac{1}{\sqrt{2\rho}} \left(\dst{\int_{\mathbb{T}^{n}} \left|\nabla_{y}\ln\po\frac{\pi_{t}}{\pi_{\infty}}\pf\right|^{2} \,\pi_{t}  }\right)^{\frac{1}{2}}.
\end{align*}

As a consequence,
\begin{align*}
\frac{\dd E}{\dd t } & \leq -\displaystyle{\int_{\mathbb{T}^{n}} \left|\nabla \ln \po \frac{\pi_{t}}{ \pi_{\infty}}\pf\right|^{2} \, \pi_{t}}+\frac{M}{\rho} \sqrt{ 1 + C e^{-4\pi^{2}t}} \left(\displaystyle{\int_{\mathbb{T}^{n}} \left|\nabla_{y}\ln \po\frac{\pi_{t}}{\pi_{\infty}}\pf\right|^{2} \,\pi_{t}  }\right)^{\frac{1}{2}}\left(\displaystyle{\int_{\mathbb{T}^{n}} \left|\nabla_{x} \ln\po\frac{\pi_{t}}{\pi_{\infty}}\pf\right|^{2} \, \pi_{t}} \right)^{\frac{1}{2}}\\
& \leq   \po - 1 + \frac{M}{2\rho} + C' e^{-2\pi^{2}t}  \pf \int_{\mathbb{T}^{n}} |\nabla \ln \po \frac{\pi_{t}}{ \pi_{\infty}}\pf|^{2} \, \pi_{t}\,.
\end{align*}
with $C'= M\sqrt{C}/(2\rho)$. Since we assumed $M<2\rho$, there exists $t_0 \ge 1$ such that for all $t \ge t_{0}$, the right hand side is negative:
$$- 1 + \frac{M}{2\rho} + C' e^{-2\pi^{2}t}:= -\alpha(t) \le 0, \quad \forall t \ge t_{0}.$$
And, given the logarithmic-Sobolec inequality of constant $R>0$ satisfied by $\pi_{\infty}$:
$$\frac{\dd E}{\dd t } \leq -2\alpha(t) R E(t) \quad \forall t \ge t_{0}.$$
Hence by Gronwall's lemma, for all $t \ge t_{0}$:
\[E(t) \ \leq \  E(t_{0}) \exp \po -2R  \int_{t_{0}}^{t}\alpha(s)  \dd s \pf  \ = \ E(t_{0}) \exp \po -2R\po 1 - \frac{M}{2\rho}\pf (t-t_0) + \frac{C'R}{2\pi^2}\pf\]

\noindent\textbf{Step 2:} As for times $t \in [0,t_{0}]$, as  in the third step of the proof of Theorem \ref{ThmConservative}, there exists $C_{2} >0$ and $C_{4} >0$ such that for all $t \in [0,t_{0}]$:
\begin{align*}
\left(\displaystyle{\int_{\mathbb{T}^{m}} |B_{t}- B_{\infty}|^{2} \, \pi^{\xi}_{t}}\right)^{\frac{1}{2}} &\leq \sqrt{\|\pi_{t}^{\xi}\|_{2}} \|B_{t}-B_{\infty} \|_{4} \ \leq \ 2C_{2}C_{4} \|\mathcal{F}\|_{\infty}.
\end{align*}
Inequality \eqref{DemThmCvtempslongBoundE} becomes, for all $t \in [0,t_{0}]$:
\begin{align*}
\frac{\dd E(t)}{\dd t} & \leq  -\dst{\int_{\mathbb{T}^{n}} \left|\nabla \ln \po\frac{\pi_{t}}{ \pi_{\infty}}\pf\right|^{2} \, \pi_{t}} +2C_{2}C_{4} \|\mathcal{F}\|_{\infty}\left(\dst{\int_{\mathbb{T}^{n}} \left|\nabla_{x} \ln\po\frac{\pi_{t}}{\pi_{\infty}}\pf\right|^{2} \, \pi_{t}} \right)^{\frac{1}{2}}\\
&\leq  C_{2}^2C_{4}^2 \|\mathcal{F}\|_{\infty}^{2}.
\end{align*}
Hence, for all $t \in [0,t_{0}]$
$$E(t) \le E(0) +\left(C_{2}C_{4} \|\mathcal{F}\|_{\infty}\right)^{2} t,$$
and
$$E(t) e^{2R\po 1-\frac{M}{2\rho} \pf t} \le \left( E(0) +\left(C_{2}C_{4} \|\mathcal{F}\|_{\infty}\right)^{2} t_{0}\right) e^{2R\po 1-\frac{M}{2\rho} \pf t_0 }$$
which concludes the proof, relying on the same argument as in the proof of Theorem~\ref{ThmConservative}.

\end{proof}
\subsection{Proof of Corollary \ref{CorCvBias}}
\label{subsec:corollary-cv}
\begin{proof}

Similarly to the previous proofs, using Lemma~\ref{lem:BoundBias} and Proposition~\ref{PropHistoUniform}, there exists $C>0$ such that, for all $t \ge 1$:
\begin{align*}
\int_{\mathbb{T}^{m}} |G_{t}-G_{\infty}|^2 \dd x & \leq  \|1/\pi_t^\xi\|_\infty \int_{\mathbb{T}^{m}} |G_{t}-G|^2 \pi_t^\xi \nonumber \\
& \leq \ (1+Ce^{-4\pi^{2}t}) \frac{2M}{\rho} E_m(t) \nonumber\\
& \leq  (1+Ce^{-4\pi^{2}t}) \frac{2M}{\rho} K e^{-\Lambda t}, 
\end{align*}
where we used either Theorem~\ref{ThmConservative} or \ref{ThmCVtempslong}. 
 For $t \in [0,1]$, we simply bound
$$\int_{\mathbb{T}^{m}} |G_{t}-G|^2 \dd x  \  \le 2 \ \|\mathcal{F}\|_{\infty}^{2} \,. $$
This concludes the ABF case, for which $B_t = G_t$ and $B_\infty = G_\infty$. Besides, the $L^2$-norm is decreased by the Helmholtz projection, which concludes the PABF case.
\end{proof}

\subsection*{Acknowledgements}

This work has received funding from the European Research Council (ERC) under the European Union’s Horizon 2020 research and innovation program (grant agreement No 810367), project EMC2. P. Monmarché acknowledges partial support by the projects EFI ANR-17-CE40-0030.

\bibliographystyle{abbrv}
\begin{small}
\bibliography{CV_PABF}
\end{small}

\end{document}